\documentclass{amsart}
\usepackage{verbatim}
\usepackage{rotating} 
\usepackage{amssymb}
\usepackage{mathrsfs,amsfonts,amsmath,amssymb,epsfig,amscd,xy,amsthm,pb-diagram} 

\newtheorem{theorem}{Theorem}[section]
\newtheorem{proposition}[theorem]{Proposition}

\newtheorem{lemma}[theorem]{Lemma}

\newtheorem{cor}[theorem]{Corollary}
\newtheorem{question}[theorem]{Question}
\newtheorem{definition}[theorem]{Definition}

\newtheorem{convention}[theorem]{Convention}
\theoremstyle{plain}

\theoremstyle{remark}

\newtheorem{remark}[theorem]{Remark}
\newtheorem{example}[theorem]{Example}

\newtheorem{assume}[theorem]{Assumption}

\newcommand{\Q}{{\mathbb Q}}

\newcommand{\Gmn}{\mathbb{G}_{\text{m}}^{n}}

\DeclareMathOperator{\oa}{oa}
\DeclareMathOperator{\pre}{pre}

\DeclareMathOperator{\Pre}{Pre}
\DeclareMathOperator{\Per}{Per}

\newcommand{\Qbar}{\bar{\Q}}

\DeclareMathOperator{\Spec}{Spec}

\DeclareMathOperator{\N}{\mathbb{N}}

\DeclareMathOperator{\id}{id}

\newcommand{\bP}{{\mathbb P}}

\newcommand{\bG}{{\mathbb G}}

\newcommand{\bA}{{\mathbb A}} \newcommand{\bQ}{{\mathbb Q}}

\newcommand{\lra}{\longrightarrow}

\newcommand{\bQb}{\bar{\bQ}}

\newcommand{\hhat}{{\widehat h}}

\newcommand{\scrS}{\mathscr{S}}

\author{D.~Ghioca}
\address{
Dragos Ghioca\\
Department of Mathematics\\
University of British Columbia\\
Vancouver, BC V6T 1Z2\\
Canada
}
\email{dghioca@math.ubc.ca}

\author{K.~D.~Nguyen}
\address{
Khoa D.~Nguyen \\
Department of Mathematics\\
University of British Columbia\\
And Pacific Institute for The Mathematical Sciences\\ 
Vancouver, BC V6T 1Z2, Canada}
\email{dknguyen@math.ubc.ca}

\thanks{The first author is partially supported by an NSERC grant. The second author is partially
	supported by a fellowship from the Pacific Institute for the Mathematical 
	Sciences.}

\begin{document}
	\title[Dynamical Anomalous Subvarieties]{Dynamical Anomalous Subvarieties: Structure and Bounded Height Theorems}

	\begin{abstract}		 
		According to Medvedev and Scanlon \cite{MS2014}, a polynomial $f(x)\in \Qbar[x]$ of degree $d\geq 2$ is called disintegrated if it is not conjugate to $x^d$ or to $\pm C_d(x)$ (where $C_d$ is the Chebyshev polynomial of degree $d$). Let $n\in\N$, let
		$f_1,\ldots,f_n\in\Qbar[x]$ be disintegrated polynomials of degrees at least 2,
		and let
		$\varphi=f_1\times\ldots\times f_n$ be the corresponding coordinate-wise self-map
		of $(\bP^1)^n$. Let $X$ be an irreducible subvariety of $(\bP^1)^n$ of
		dimension $r$ defined over $\Qbar$. We define the \emph{$\varphi$-anomalous} locus of $X$ which is  related to the \emph{$\varphi$-periodic}
		subvarieties of $(\bP^1)^n$. We prove that the $\varphi$-anomalous locus of $X$ is Zariski closed; this is a dynamical analogue of a theorem of Bombieri, Masser, and Zannier \cite{BMZ07}.  We also prove that the points in the intersection of $X$ with the union of all irreducible $\varphi$-periodic subvarieties of $(\bP^1)^n$ of codimension $r$ have bounded height outside the $\varphi$-anomalous locus of $X$; this is  a dynamical analogue of  Habegger's theorem \cite{Habegger09} which was previously 
		conjectured in \cite{BMZ07}. The slightly more general self-maps $\varphi=f_1\times\ldots\times f_n$
		where each $f_i\in \Qbar(x)$ is a disintegrated rational map are also treated at the end of the paper.
	\end{abstract}
	\maketitle

	\section{Introduction}\label{sec:intro}
	Throughout this paper, by a (not necessarily irreducible) variety, we mean a reduced scheme
	of finite type \emph{over $\Qbar$}; and by a subvariety we mean a reduced \emph{closed}
	subscheme. For a map $\mu$
	from a set to itself and for every positive integer $m$,
	we let $\mu^m$ denote the $m$-fold iterate: $\mu\circ\ldots\circ\mu$; the notation $\mu^0$ denotes the identity map. 	
	Let $h$ 
	denote the absolute logarithmic Weil height on $\bP^1(\Qbar)$ (see \cite[Chapter 1]{BG} 
	or \cite[Part B]{HS}). Let $n$ be a positive 
	integer,
	we define the height function $h_n$ on $(\bP^1)^n(\Qbar)$ by $h_n(a_1,\ldots,a_n):=h(a_1)+\ldots+h(a_n)$. 
	When we say that a subset of $(\bP^1)^n(\Qbar)$ has bounded height, we mean boundedness
	with respect to $h_n$. 
	

	After a series of papers \cite{BMZ_TAMS06}, \cite{BMZ_Scuola08} and \cite{BMZ07} following the seminal work \cite[Theorem 1]{BMZ99},
	Bombieri, Masser and Zannier define anomalous subvarieties
	in $\Gmn$ as follows. By a \emph{special} subvariety of $\Gmn$, we mean a 
	translate  of an irreducible algebraic subgroup. 
	For any  irreducible subvariety $X\subseteq \Gmn$ of dimension $r$, an
	irreducible subvariety $Y$ of $X$ is said to be \emph{anomalous} 
	(or better, $X$-anomalous) if there exists
	a special subvariety $Z$ satisfying the following conditions:
	\begin{equation}
\label{definition anomalous classic}
Y\subseteq X\cap Z\text{ and }\dim(Y) > \max\{0,\dim(X)+\dim(Z)-n\}.
	\end{equation}
	We define $X^{\oa}:=X\setminus \bigcup_Y Y$, 
	where $Y$ ranges over all anomalous subvarieties of $X$. We let $\mathbb{G}_m^{n,[r]}$ be the union of all algebraic subgroups of $\Gmn$ of \emph{codimension} $r$. The following has been established by
	Bombieri, Masser, Zannier \cite[Theorem~1.4]{BMZ07} and Habegger \cite[Theorem~1.2]{Habegger09} (after being previously conjectured in \cite{BMZ07}):
	
	\begin{theorem}\label{thm:2 in 1}  
	Let $X$ be an irreducible
	subvariety of $\Gmn$ of dimension $r$ (defined over $\Qbar$, as always). We have:
	\begin{itemize} 
	\item[(a)] \rm{(Bombieri-Masser-Zannier)} Structure Theorem: the set $X^{\oa}$ is Zariski open in $X$. Moreover, there exists a finite collection $\mathcal{T}$ of subtori
	of $\Gmn$ (depending on $X$) such that the anomalous locus of $X$ is the union of all  anomalous subvarieties $Y$ of $X$ for which 
	there exists  a translate $Z$ of a tori in $\mathcal{T}$
	satisfying $Y\subseteq X\cap Z$ and $\dim(Y)>\max\{0,\dim(X)+\dim(Z)-n\}$. 
	\item[(b)] \rm{(Habegger)} Bounded Height Theorem: the
	set $X^{\oa}\cap 
	\mathbb{G}_m^{n,[r]}$ has bounded height.
	\end{itemize}
	\end{theorem}
	
	The Bounded Height Theorem is closely related to the problem of unlikely intersections in arithmetic geometry introduced in \cite{BMZ99} whose motivation comes from the classical Manin-Mumford conjecture (which is Raynaud's theorem \cite{Raynaud, Raynaud2} for abelian varieties and Laurent's theorem \cite{Laurent}  for $\Gmn$). Moreover, Pink \cite{Pink}
	and Zilber \cite{Zilber} independently propose a similar problem to the unlikely intersection problem introduced in \cite{BMZ99} in the more general
	context of semiabelian varieties and mixed Shimura varieties. For an
	excellent treatment of these topics, we refer the readers to Zannier's book \cite{Zan}.
	Both the Bounded Height Theorem and the Pink-Zilber problem are also considered in the context of function
	fields in \cite{CGMM}.
	On the other hand, very little is known in the context of arithmetic dynamics. Zhang 
	\cite{ZhangDist}  proposed a dynamical analogue of the Manin-Mumford conjecture, which was later amended in \cite{GTZ} and more recently in \cite{Yuan-Zhang}. 
	
	This paper is the first to establish a dynamical analogue of Theorem~\ref{thm:2 in 1}. 
	For $d\geq 2$, the Chebyshev polynomial
	$C_d$ is the unique polynomial of degree $d$ such that    $C_d\left(x+\frac{1}{x}\right)=
	x^d+\frac{1}{x^d}$. Following the terminology in Medvedev-Scanlon 
	\cite{MS2014},
	we say that a polynomial $f\in\Qbar[x]$ of degree $d\geq  2$ 
	is \emph{disintegrated} if it is not linearly conjugate
	to $x^d$ or $\pm C_d(x)$. Let $n\geq 2$ and let $f_1,\ldots,f_n\in\Qbar[x]$ be polynomials
	of degrees at least 2. Let $\varphi=f_1\times\ldots\times f_n$
	be the induced coordinate-wise self-map of $(\bP^1)^n$. According to \cite[Theorem~2.30]{MS2014}, 
	to study
	the arithmetic dynamics of $\varphi$, it suffices to study 
	two cases: the case when all the $f_i$'s are not disintegrated which reduces to 
	diophantine questions on $\Gmn$ (as studied by Bombieri-Masser-Zannier \cite{BMZ99, BMZ_TAMS06, BMZ07, BMZ_Scuola08}) and the case when all the $f_i$'s are disintegrated.

	It is the dynamics of $\varphi$ in the case where $f_i$ is
	disintegrated for $1\leq i\leq n$ for which we prove an analogue of Theorem~\ref{thm:2 in 1}. In 
	fact, one of the main theorems of \cite{IMRN2013} provides
	a bounded height result when we intersect a fixed \emph{curve} with 
	\emph{periodic 
	hypersurfaces}. Our main results in this paper (see Theorem~\ref{thm:informal}
	and Theorem~\ref{thm:open}) not only solve
	 completely a bounded height problem for a general ``dynamical complementary
	 dimensional intersection'' (similar to Theorem~\ref{thm:2 in 1} (b)), but also establish a structure theorem
	 similar to Theorem~\ref{thm:2 in 1} (a).
	
	An irreducible subvariety $V$ of $(\bP^1)^n$ is said to be 
	\emph{periodic} (or better, $\varphi$-periodic)
	if there exists an integer $m>0$ such that
	$\varphi^{m}(V)=V$. If
	there exists $k\geq 0$ such that $\varphi^k(V)$
	is periodic then we say that $V$ is \emph{preperiodic} (or
	better, $\varphi$-preperiodic). 
	While it is most natural to regard periodic subvarieties as
	a dynamical analogue of \emph{irreducible} algebraic subgroups, the
	first major obstacle is to come up with an analogue of
	\emph{arbitrary translates} of subgroups (which were the special subvarieties of $\Gmn$). Motivated by \cite[Theorem 1.2]{IMRN2013},
	we let $C$ be the curve $\zeta\times \bP^1$ in $(\bP^1)^2$ (endowed with the diagonal action of a polynomial $f\in\Qbar[z]$) and
	we intersect $C$ with periodic hypersurfaces defined by $y=f^{\ell}(x)$
	for $\ell\geq 0$; then the resulting set will have unbounded height (if $\zeta$ is not preperiodic). Hence, in order to establish a dynamical analogue of Theorem~\ref{thm:2 in 1}, we have to exclude certain varieties having a constant projection
	to some factor $(\bP^1)^m$ of $(\bP^1)^n$.   
	
	
	Let $I:=\{1,\ldots,n\}$. For every non-empty subset $J$ of 
	$I$, let $\varphi^J$ be the 
	coordinate-wise self-map of $(\bP^1)^{|J|}$ induced by 
	the polynomials $f_j$ for $j\in J$.
  We say that an irreducible subvariety $Z$ of $(\bP^1)^n$ is 
	$\varphi$-special if there is a subset $J$ of $I$ such that after a possible rearrangement of the factors of 
	$(\bP^1)^n$, $Z$ has the form $\zeta\times V$, where $\zeta\in 
	(\bP^1)^{n-|J|}(\Qbar)$ and $V\subseteq (\bP^1)^{|J|}$  is $\varphi^J$-periodic.  We refer the readers to Definition~\ref{def:dyn special} for a formal definition of special subvarieties for the dynamics of $\varphi$. 
	We remark that we got our inspiration for defining the $\varphi$-special subvarieties from the classical case of special subvarieties of $\bG_m^n$ (see \cite{BMZ99}) since the irreducible periodic subvarieties  are the analogue of irreducible algebraic groups, while $\zeta\times V$ we consider to be the analogue of a \emph{translation}.  
We also note that using the above definition for $\varphi$-special subvarieties  when each $f_i(x)$ is the same power map (i.e. $f_i(x)=x^d$ for some $d\geq 2$), we get arbitrary translates of subtori (since $\zeta$ need not be a root of unity); 
	 yet some translates of tori are not $\varphi$-special as defined above. 
    This should not
		be a surprise though since it is well-known that $x^d$ (and $\pm C_d(x)$)
		have very different dynamical behavior compared to disintegrated
		polynomials. For example, while a subtorus of codimension $1$ can be 
		described by an equation involving every variable, a theorem of
		Medvedev-Scanlon (see Section~\ref{sec:MS theorem}) asserts that
		when each $f_i(x)$ is disintegrated, every $\varphi$-periodic hypersurface
		is described by an equation involving only at most two variables. Next we define a dynamical 
		analogue of anomalous subvarieties and of the set $X^{\oa}$:

\begin{definition}\label{def:dyn anomalous}
Let $n\geq 2$, let $f_1(x),\ldots,f_n(x)\in \Qbar[x]$ be disintegrated polynomials of degrees at least 2, 
and let $\varphi:=f_1\times\ldots\times f_n$ be as before. For an irreducible subvariety $X\subseteq (\bP^1)^n$, an irreducible subvariety $Y\subseteq X$ is called $\varphi$-anomalous (with respect to $X$) 
if there exists an irreducible $\varphi$-special subvariety $Z\subseteq (\bP^1)^n$ such that
\begin{equation}
\label{definition anomalous}
Y\subseteq X\cap Z\text{ and }\dim(Y) > \max\{0,\dim(X)+\dim(Z)-n\}.
	\end{equation}
%
Define $X^{\oa}_{\varphi}$ to be the complement in $X$ of the union of all  $\varphi_n$-anomalous subvarieties of $X$. 
\end{definition}
For each $0\leq r\le n$, we let $\Per_{\varphi}^{[r]}$ be the union of all irreducible $\varphi$-periodic subvarieties of $(\bP^1)^n$ of \emph{codimension} $r$.
When the map $\varphi$ (hence the collection $\{f_1,\ldots,f_n\}$) is clear from the context, 
we will use the notation $\Per^{[r]}$. Our first main result is the following
dynamical analogue of the Bounded Height Theorem:

\begin{theorem}\label{thm:informal}
		Let $n\geq 2$, let $f_1(x),\ldots,f_n(x)\in \Qbar[x]$ be disintegrated polynomials of degrees at least 2, 
and let $\varphi:=f_1\times\ldots\times f_n$ be as before.
		Let $X$ be an irreducible subvariety of $(\bP^1)^n$ of dimension $r$.  
    Then the set $X^{\oa}_{\varphi}\cap \Per^{[r]}$ has bounded height.
\end{theorem}
	 
We give examples at the end of Section~\ref{sec:MS theorem} that it is 
necessary to 
omit the anomalous subvarieties of $X$. Theorem~\ref{thm:informal}
generalizes the results in \cite{IMRN2013} which only
treat the case that $X$ is a curve. Our second main result is
a dynamical analogue of the Structure Theorem:
\begin{theorem}\label{thm:open}
	With the notation as in Theorem~\ref{thm:informal}, we have that the set $X^{\oa}_{\varphi}$ is Zariski open in $X$. 

Also, there exists a collection (depending on $\varphi$ and $X$) consisting of finitely many (not necessarily distinct) subsets $J_1,\ldots,J_{\ell}$ of $I=\{1,\ldots,n\}$ 
together with $\varphi^{J_k}$-periodic subvarieties $Z_k\subseteq (\bP^1)^{|J_k|}$
for $1\leq k\leq \ell$
such that the following holds. The $\varphi$-anomalous locus of $X$ is the union of all anomalous subvarieties $Y$ for which, 
after a possible rearrangement of coordinates, there is a special subvariety $Z$ of the form $\zeta\times Z_k$  for some $1\leq k\leq \ell$ and $\zeta\in (\bP^1)^{n-|J_k|}(\Qbar)$ such that
$Y\subseteq X\cap Z$ and $\dim(Y)>\max\{0,\dim(X)+\dim(Z)-n\}$. 
\end{theorem}

\begin{remark}
\label{effective remark}
As pointed out by the referee, the height bound in Theorem~\ref{thm:informal} and the equations defining the subvarieties $Z_i$ in Theorem~\ref{thm:open} are effectively computable since our arguments are effective.
\end{remark}

It is also possible to ask a variant of the above theorems for preperiodic
subvarieties. We define $\varphi$-pre-special subvarieties 
to be those of the form $\zeta\times Z$ where $\zeta\in (\bP^1)^{n-|J|}(\Qbar)$
and $Z\subseteq (\bP^1)^{|J|}$
is $\varphi^{J}$-preperiodic for some subset $J$ of $\{1,\ldots,n\}$. For an irreducible subvariety
$X\subseteq (\bP^1)^n$, we define $\varphi$-pre-anomalous subvarieties as
in Definition~\ref{def:dyn anomalous} where the only change is that $Z$ is required
to be $\varphi$-pre-special rather than $\varphi$-special. Similarly,
we define $X^{\oa,\pre}_{\varphi}$ to be
the complement of the union of all $\varphi$-pre-anomalous 
subvarieties in $X$.
Finally, we 
use the notation $\Pre_{\varphi}^{[r]}$
(or $\Pre^{[r]}$ if $\varphi$ is clear)
to denote the union of all $\varphi$-preperiodic
subvarieties of codimension $r$. 
We expect the following to have an affirmative answer:
\begin{question}\label{q:preperiodic case}
	Let $n$, $f_1,\ldots,f_n$, $\varphi$, $X$ and $r$ be as in Theorem~\ref{thm:informal}.
	\begin{itemize}
		\item [(a)] Does the set $X^{\oa,\pre}_{\varphi}\cap \Pre^{[r]}$ have
		bounded height?
		\item [(b)] Is the set $X^{\oa,\pre}_{\varphi}$
		 Zariski open in $X$?
	\end{itemize} 
\end{question}

Note that Question~\ref{q:preperiodic case} is \emph{neither stronger nor weaker}
than our main theorems. The set $\Pre^{[r]}$ is larger than $\Per^{[r]}$, yet
the set $X^{\oa,\pre}_{\varphi}$ is
smaller than $X^{\oa}_{\varphi}$. Note that 
the largest intersection 
$X^{\oa}_{\varphi}\cap \Pre^{[r]}$ cannot have bounded height. 
As an easy example, we let $n=2$, $f_1=f_2=:f$, and let $X$ be a periodic but not preperiodic curve in $(\bP^1)^2$
having non-constant projection to each factor $\bP^1$ (for instance,
an irreducible component of the curve $f(x)=f(y)$ other than the diagonal $x=y$). In this case,
we have $X^{\oa}_{\varphi}=X\subset \Pre^{[r]}$.


The two main 
	ingredients for the
	proof of Theorem~\ref{thm:informal} is the classification of $\varphi$-periodic
	subvarieties by Medvedev-Scanlon presented in the next section and 
	an elementary height inequality (see Corollary~\ref{cor:key inequality}). The proof of 
	Theorem~\ref{thm:open} also uses certain bounded
	height arguments that are similar to those in the proof of Theorem~\ref{thm:informal}. 
	One natural way to attack Question~\ref{q:preperiodic case} is to
	relate it to results in the periodic case such as Theorem~\ref{thm:informal}
	and Theorem~\ref{thm:open} since for every preperiodic subvariety $V$ we have
	$\varphi^k(V)$ is periodic for some $k$. The difficulty is to obtain a good 
	bound
	(perhaps uniform in the
	height of $X$) on the inequalities obtained in the proof of Theorems \ref{thm:informal}
	and \ref{thm:open}. For more details when $X$ is a curve, we refer the readers to
	\cite{IMRN2013}.
	

In this paper, we first provide all the details for the proof
of Theorem~\ref{thm:informal} and Theorem~\ref{thm:open}
in the special yet essential case of the self-map
$f\times\ldots\times f$ on $(\bP^1)^n$ (i.e. when $f_1=\ldots=f_n=f$). The general case $\varphi=f_1\times\ldots f_n$ is explained
in Section~\ref{sec:general} and consists of two steps. First, we reduce
 $\varphi$ to a map of the form $\psi=\psi_1\times\ldots\times\psi_s$ 
for some $1\leq s\leq n$ where $\psi_i$
is a coordinate-wise self-map of $(\bP^1)^{n_i}$ of the form
$w_i\times\ldots\times w_i$ for some $1\leq n_i\leq n$ and disintegrated
$w_i(x)\in\Qbar[x]$ such that $\sum_{i=1}^s n_i=n$ and $w_i$ and $w_j$ are
``inequivalent'' for $i\neq j$ (see Definition~\ref{def:relation}
and the discussion following Definition~\ref{def:relation}). Then maps $\psi$  (as above) could be treated by a completely similar arguments used 
to settle the special case $\varphi=f\times\ldots\times f$ albeit with a more
complicated system of notation for bookkeeping.    

%
%

In fact, we can define the sets $X^{\oa}_{\phi}$
  and $X^{\oa,\pre}_{\phi}$
  and formulate the dynamical bounded height and structure theorems for 
  the more general self-map of $(\bP^1)^n$
  of the form $\phi=f_1\times\ldots\times f_n$
  where each $f_j(x)\in \Qbar(x)$ is 
  a ``disintegrated \emph{rational} function''. The readers are referred to 
  Question~\ref{q:general} and the Remarks~\ref{rem:rational f} and \ref{last general remark} for more details.
In view of our main results, another possible direction of research is to formulate a weak dynamical form of
the classical Pink-Zilber Conjecture by asking that if $X\subseteq (\bP^1)^n$ is not contained
in a $\varphi$-periodic (resp. $\varphi$-preperiodic) hypersurface, then $X \cap \Per^{[r+1]}$ (resp. $X\cap \Pre^{[r+1]}$) is not Zariski dense in $X$, 
where $r=\dim(X)$. In the
case of hypersurfaces in $(\bP^1)^n$, this restricts to the dynamical Manin-Mumford problem
formulated in \cite{ZhangDist}. The case of lines $X\subseteq (\bP^1)^2$ was already proven
in \cite{GTZ}, and in light of the observations made in \cite[Section~3]{Yuan-Zhang}, one might
expect that the general case of curves, and perhaps even the case of arbitrary subvarieties
of $(\bP^1)^n$ holds when each coordinate of $\bP^1$ is acted on by a \emph{disintegrated polynomial}.
However, the scarcity of positive results for the Dynamical Manin-Mumford
Conjecture and also, the exotic behavior of certain examples produced in \cite{GTZ} 
when different \emph{rational functions} $f_i$ act on the coordinates of $(\bP^1)^n$ prevent
us from formally stating a conjecture for the unlikely intersection principle in dynamics.

	\emph{For the rest of this paper, we do not refer to the notion of special and anomalous subvarieties of $\Gmn$ given by Bombieri, Masser and Zannier. For simplicity, when the self-map $\varphi$ is clear from the context,
	we use the terminology special (resp. pre-special, anomalous, pre-anomalous)
	subvarieties instead of $\varphi$-special (resp. $\varphi$-pre-special, $\varphi$-anomalous, $\varphi$-pre-anomalous) subvarieties.}

	The organization of this paper is as follows. In Section~\ref{sec:MS theorem}, following \cite{MS2014} and \cite{IMRN2013} we give a precise description of the $\varphi$-periodic subvarieties of $(\bP^1)^n$ in the special
	case $f_1=\ldots=f_n$. The proof of Theorems 
	\ref{thm:informal} and \ref{thm:open} in that case takes up the following four sections. This 
	proof uses 
	certain properties of height and canonical height in Section~\ref{sec:height} 
	coupled with some elementary geometric properties
	of the set $X^{\oa}_{\varphi_n}$ in Section~\ref{sec:elementary geometry}. In the
	last section, we explain how to prove Theorem~\ref{thm:informal}
	and Theorem~\ref{thm:open} in general and 
	conclude the paper with a brief discussion on the dynamics of more 
	general self-maps of $(\bP^1)^n$ of the form
	$f_1\times\ldots\times f_n$ where each $f_i(x)\in\Qbar(x)$ is a ``disintegrated rational
	map''.

	{\bf Acknowledgments.}  We thank Tom Scanlon and
	Tom Tucker for many useful conversations. We are grateful to the anonymous referee
	for many useful suggestions.

	
 \section{Structure of periodic subvarieties}\label{sec:MS theorem}	

Recall that a polynomial $f\in \Qbar[x]$ of degree $d\geq 2$ is called disintegrated
if it is not linearly conjugate to $x^d$ or $\pm C_d(x)$. Let $n$ be a positive integer and
$f_1,\ldots,f_n\in\Qbar[x]$ be disintegrated polynomials of degrees at least 2. Let
$\varphi=f_1\times\ldots\times f_n$ be the corresponding coordinate-wise self-map
of $(\bP^1)^n$.
Let $I_n:=\{1,\ldots,n\}$. For each \emph{ordered}
	subset $J$ of $I_n$, we define:
	$$(\bP^1)^J:=(\bP^1)^{|J|}$$  
	equipped with the canonical projection $\pi^J:\ (\bP^1)^n\rightarrow (\bP^1)^J$. 
	Occasionally, we also work with the Zariski open subset $(\bA^1)^n=\bA^n$
	of $(\bP^1)^n$ and we use the notation $\bA^J$ to denote the Zariski
	open subset $(\bA^1)^{|J|}$ of $(\bP^1)^J$. Obviously, $\bA^J=\pi^J(\bA^n)$. 
	Let $\varphi^J$ denote the coordinate-wise self-map of $(\bP^1)^J$
	induced by the polynomials $f_j$'s for $j\in J$.
	In this paper,
	we will consider ordered subsets of $I_n$ whose orders need not be induced from the 
	usual order
	of the set of integers.
	If $J_1,\ldots,J_m$
	are ordered subsets of $I_n$ which partition $I_n$, then
	we have the canonical isomorphism:
	$$(\pi^{J_1},\ldots,\pi^{J_m}):\ (\bP^1)^n=(\bP^1)^{J_1}\times\ldots\times (\bP^1)^{J_m}.$$
	For each irreducible subvariety $V$ of $(\bP^1)^n$, let $J_V$ denote the 
	set of all $j\in I_n$ such that the projection from $V$ to the $j^\text{th}$ coordinate $\bP^1$
	is constant. If $J_V\neq \emptyset$, we equip $J_V$  with the natural order of the set of 
	integers, and we let $a_V\in (\bP^1)^{J_V}(\Qbar)$ denote $\pi^{J_V}(V)$.
	Even when $J_V=\emptyset$, we will \emph{vacuously} define $(\bP^1)^{J_V}$ as
	the scheme of one point $\Spec(\Qbar)$ and define $a_V$ to be that point.
	We have then the following formal definition for \emph{special} subvarieties.  
	\begin{definition}\label{def:dyn special}
	   Let $Z$ be an irreducible subvariety of $(\bP^1)^n$. Identify
	   $(\bP^1)^n=(\bP^1)^{J_Z}\times (\bP^1)^{I_n-J_Z}$. We say that $Z$ is $\varphi$-special 
	   (respectively $\varphi$-pre-special) if it has the form $a_Z\times Z'$
	   	where $Z'$ is $\varphi^{I_n-J_Z}$-periodic (respectively preperiodic).
	\end{definition}
	

	We now present (a crucial case of) the Medvedev-Scanlon
	classification of $\varphi$-periodic subvarieties of $(\bP^1)^n$ from 
	\cite{MS2014} along with its refinement from \cite{IMRN2013}. The complete
	classification in \cite{MS2014} consists of two parts: the first part treats the special 
	case 
	$f_1=\ldots=f_n$ which is given in this section while the second part
	explains why the general case reduces to this special case and is given in
	Section~\ref{sec:general}. We have the following:
	
	\begin{assume}\label{assume:same maps}
	For the rest of this section, assume that $f_1(x)=\ldots=f_n(x)$ and denote this
	common polynomial by $f(x)$. Since the
	map $\varphi^J$ on $(\bP^1)^{J}$ now depends
	only on $|J|$, for every positive integer $m$ we usually use the
	notation 
	$\varphi_m:=\varphi_{m,f}$ to denote the 
	 self-map $f\times\ldots\times f$ on $(\bP^1)^m$. In particular,
	 $\varphi=\varphi_n$ and $\varphi^J=\varphi_{|J|}$; such notation also emphasizes the
	 fact that we are assuming $f_1=\ldots=f_n=f$.
	\end{assume}

	 Let $x_1,\ldots,x_n$ denote the coordinate functions of the factors $\bP^1$ of 
	 $(\bP^1)^n$. Medvedev and Scanlon prove the following important result 
	 \cite[pp.~85]{MS2014}: 
				 \begin{theorem}\label{thm:MS theorem}
				 	Let $V$ be an irreducible $\varphi_n$-periodic  
				 	subvariety of $(\bP^1)^n$.  
				 	Then
				 	$V$ is given by a collection of equations of the following forms: 
				 		\begin{itemize}
				 			\item [(a)] $x_i=\zeta$ where $\zeta$
				 			is a periodic point of $f$.
				 			\item [(b)] $x_j=g(x_i)$ for some $1\leq i\neq j\leq n$, 
				 			where $g$ is a polynomial commuting with an iterate of $f$. 
				 		\end{itemize}
				 \end{theorem}

	Next we describe all polynomials $g$ commuting with an
	iterate of $f$ mentioned in Theorem~\ref{thm:MS theorem}; the following result is \cite[Proposition~2.3]{IMRN2013}. 
				\begin{proposition}\label{prop:all_g}
					We have:
					\begin{itemize}
						\item [(a)] If $g\in \Qbar[x]$ has degree at least $2$ such that 
						$g$
						commutes with an iterate of $f$ then $g$ and $f$
						have a common iterate.
						\item [(b)] Let $M(f^\infty)$
						denote the collection of all linear polynomials
						commuting with an iterate of $f$. Then $M(f^\infty)$
						is a finite cyclic group under composition.
						\item [(c)] Let $\tilde{f}\in \Qbar[x]$ be a polynomial
						of minimum degree $\tilde{d}\ge 2$   such that $\tilde{f}$
						commutes with an iterate of $f$. Then there exists
						$D=D_f>0$ relatively prime to
						the order of $M(f^\infty)$ such that 
						$\tilde{f}\circ L=L^D\circ \tilde{f}$
						for every $L\in M(f^\infty)$.
						\item [(d)] 
						$\left\{\tilde{f}^m\circ L:\ m\geq 0,\ L\in M(f^\infty)\right\}=\left\{L\circ\tilde{f}^m:\ m\geq 0,\ L\in M(f^\infty)\right\}$, and these sets describe exactly all
						polynomials $g$ commuting with an iterate of $f$. As a consequence,
						there are only finitely many polynomials of bounded degree commuting
						with an iterate of $f$.		
					\end{itemize}
				\end{proposition}

  From the discussion in \cite[Proposition 2.3]{IMRN2013},
  we have the following more refined description of $\varphi_n$-periodic
  subvarieties:  	
	\begin{proposition}\label{prop:refined MS}
	  \begin{itemize}
	  	\item [(a)] Let $V$ be an irreducible $\varphi_n$-periodic subvariety of $(\bP^1)^n$
	  	of dimension $n-r$. Then there exists a partition of $I_n-J_V$
	  	into $n-r$ non-empty subsets $J_1,\ldots,J_{n-r}$ such that
	  	the following hold.
	  	We fix an order on each $J_1,\ldots,J_{n-r}$, and identify:
	  		$$(\bP^1)^n=(\bP^1)^{J_V}\times(\bP^1)^{J_1}\times\ldots\times(\bP^1)^{J_{n-r}}.$$
	    For $1\leq i\leq n-r$, there exists an irreducible $\varphi_{|J_i|}$-periodic
	    curve $C_i$ in $(\bP^1)^{J_i}$
	    such that:
	    $$V=a_V\times C_1\times\ldots\times C_{n-r}.$$
	    
	    \item [(b)] Let $C$ be an irreducible $\varphi_n$-periodic curve
	    in $(\bP^1)^n$ and denote $m:=|I_n-J_C|\geq 1$. Then there exist
	    a permutation $(i_1,\ldots,i_m)$
	    of $I_n-J_C$ and non-constant polynomials 
	    $g_{i_2},\ldots,g_{i_m}\in \bQb[x]$
	    such that $C$ is given by the equations
	    $x_{i_2}=g_{i_2}(x_{i_1})$,\ldots,
	    $x_{i_m}=g_{i_m}(x_{i_{m-1}})$. Furthermore,
	    the polynomials $g_{i_2},\ldots,g_{i_m}$
	    commute with an iterate of $f$.
	  \end{itemize}	
	\end{proposition}

	\begin{remark}\label{rem:get order}
		The permutation $(i_1,\ldots,i_m)$ mentioned in part (b) of Proposition
		\ref{prop:refined MS}
		induces the order $i_1\prec \ldots\prec i_m$ on $I_n-J_C$. Such
		a permutation and its induced order are not uniquely determined by $V$.
		For example, let $L$ be a \emph{linear} polynomial commuting with an iterate of 
		$f$.
		Let $C$ be the periodic curve in $(\bP^1)^2$ defined by the equation 
		$x_2=L(x_1)$. Then $I-J_C=\{1,2\}$, and $1\prec 2$ is an order satisfying the conclusion of 
		part (b).
		However, we can also express $C$ as $x_1=L^{-1}(x_2)$. Then the order $2\prec 1$ 
		also satisfies part (b). Therefore, in part (a), the choice of an
		order on each $J_i$ is not unique. Nevertheless, the partition of $I_n-J_V$
		into the subsets $J_1,\ldots,J_r$ is unique (see \cite[Section 2]{IMRN2013}).
	\end{remark}

	\begin{definition}\label{def:signature}
Let $V$ be an irreducible $\varphi_n$-periodic subvariety of dimension $r$. Partition
		$I_n-J_V$ into $J_1,\ldots,J_r$, and let $C_1,\ldots,C_r$
		be the periodic curves as in the conclusion of part (a) of Proposition \ref{prop:refined MS}.
		For $1\leq i\leq r$, we equip
		$J_i$ with an order satisfying the conclusion of part (b) of Proposition \ref{prop:refined MS}
		for the curve $C_i$
		and Remark \ref{rem:get order}. Then the collection consisting of $J_V$ and the 
		ordered
		sets $J_1,\ldots,J_r$ is called a signature of $V$.
	\end{definition}

	\begin{remark}
\label{fixed signature}
		Using the fact that there are finitely many signatures for periodic varieties, in Theorem~\ref{thm:informal} it suffices to show that for any given signature $\scrS$, the intersection of $X^{\oa}_{\varphi_n}$ with the union of all periodic subvarieties having signature $\scrS$ and codimension $r$ has bounded height.  
	\end{remark}

  We finish this section by giving a couple of examples to show
	why it is necessary to remove the anomalous locus
	in Theorem~\ref{thm:informal}.
	\begin{example}
\label{example 1}
Let $X\subseteq (\bP^1)^5$ be a $3$-fold with the property that its intersection with the periodic $3$-fold $V$ given by the equations $x_2=f(x_1)$ and $x_3=f(x_2)$ contains a surface $S$.  We claim that $S$ should be removed from $X$ in order for the points in the intersection with $\Per^{[3]}$ have bounded height. Indeed, for each positive integer $m$, we let $V_m$ be the periodic surface given by the equations
$$x_2=f(x_1)\text{, }x_3=f(x_2)\text{ and }x_4=f^m(x_3).$$
Then $V_m\subseteq V$ and so, $H_m\cap S\subseteq V_m\cap X$ where $H_m\subseteq (\bP^1)^5$ is the hypersurface given by $x_4=f^m(x_3)$. In particular, there is a curve $C_m\subseteq H_m\cap S\subseteq V_m\cap X$, and obviously this curve contains points of arbitrarily large height. It is easy to see that each curve $C_m$ is different as we vary $m$, and moreover, their union is Zariski dense in $S$. 
\end{example}

\begin{example}
\label{example 2}
Let $X\subseteq (\bP^1)^4$ be a surface with the property that its intersection with the surface $(a_1,a_2)\times (\bP^1)^2$ (for $a_1,a_2\in \bP^1(\Qbar)$) contains a curve $C$. Assume $a_1$ is not preperiodic,  $a_2=f(a_1)$, and also that $C$ projects  onto the third  coordinate of $(\bP^1)^4$. We show that the curve $C$ must be removed from $X$ in order for the points in the intersection with $\Per^{[2]}$ have bounded height. Indeed, for each positive integer $m$, we let $V_m$ be the periodic surface given by the equations: $x_2=f(x_1)$ and $x_3=f^m(x_2)$. Because $C$ projects onto the third coordinate of $(\bP^1)^4$, there exists $a_4\in\bP^1(\Qbar)$ such that $\left(a_1,f(a_1),f^{m+1}(a_1),a_4\right)\in C\cap V_m\subseteq X\cap V_m$, whose height grows to infinity as $m\to\infty$ (because $a_1$ is not preperiodic).

As an aside, we note that even though the above anomalous curves  need to be removed, it is not clear whether one would also have to remove the curves which appear in the intersection of  a surface $X\subseteq (\bP^1)^4$ with $(a_1,a_2)\times (\bP^1)^2$ if $a_1$ and $a_2$ are in different orbits under $f$. This phenomenon also occurs in the diophantine situation
(see the discussion in \cite[Section~5, pp.~24-25]{BMZ07}).
\end{example}

	\section{Properties of the height}\label{sec:height}

	Recall that $h$ denotes the absolute logarithmic Weil
	height on $\bP^1(\Qbar)$ (see \cite[Part B]{HS} or \cite[Chapter 1]{BG}). The following result is well-known:
	\begin{lemma}\label{lem:easyheight}
		For every $a,b\in \bQb$, we have:
		\begin{itemize}
			\item [(a)] $h(ab)\leq h(a)+h(b)$.
			\item [(b)] $h(a)-h(b)\leq h(a/b)$ if $b\neq 0$.
			\item [(c)] $h(a+b)\leq h(a)+h(b)+\log2$.
			\item [(d)] $h(a)-h(b)-\log2\leq h(a-b)$.
\item[(e)] $h(a^d)=|d|\cdot h(a)$ for any integer $d$ (where $a\ne 0$ if $d<0$).
		\end{itemize}
	\end{lemma} 
	

	We can use Lemma~\ref{lem:easyheight} to prove the following result (we note that part (b) of the following Lemma is instrumental in our proof of Theorem~\ref{thm:informal}).
	\begin{lemma}\label{lem:important}
		Let $n\geq 1$ and  $F(X_1,\ldots,X_n)\in \bQb[X_1,\ldots,X_n]$ having degree $D\geq 1$ in $X_n$.
		Write: 	$$F(X_1,\ldots,X_n)=F_D(X_1,\ldots,X_{n-1})X_n^D+\ldots+F_0(X_1,\ldots,X_{n-1}).$$
For each $i=1,\dots, n$, let $D_i$ be the degree in $X_i$ of $F$ (so $D_n=D$).   
		The following hold.
		\begin{itemize}
			\item [(a)]  Then there  exists a positive constant $C_1$ depending only
			on $F(X_1,\ldots,X_n)$ such that for every
			$a_1,\ldots,a_n\in \bQb$, we have
			  $$h(F(a_1,\ldots,a_n))\leq \sum_{i=1}^n D_ih(a_i)+C_1.$$
			  
			\item [(b)] There exists a positive constant $C_2$ depending
			only on $F(X_1,\ldots,X_n)$ such that
			for every $a_1,\ldots,a_n \in\bQb$ satisfying $F_i(a_1,\ldots,a_{n-1})\neq 0$
			for some $1\leq i\leq D$,
			we have:
			$$h(a_n)-2D_1h(a_1)-\ldots-2D_{n-1}h(a_{n-1})-C_2\leq h(F(a_1,\ldots,a_n)).$$ 
		\end{itemize}
	\end{lemma}
	\begin{proof}
\begin{enumerate}
\item[(a)] This is a standard result which follows using an analysis at each place (both archimedean and nonarchimedean) similar to the proof of  Lemma~\ref{lem:easyheight}.
\item[(b)] Given $a_1,\dots, a_n\in\Qbar$, we let $i$ be maximum with the property that $F_i(a_1,\dots, a_{n-1})\ne 0$. Then, letting 
$$Q_i(x_1,\dots, x_n):=F_{i-1}X_n^{i-1}+\ldots+F_0$$
so that $F(a_1,\ldots,a_n)=F_i(a_1,\ldots,a_{n-1})a_n^i+Q_i(a_1,\ldots,a_n)$ we obtain the following inequalities from part (a) and Lemma~\ref{lem:easyheight}: 
\begin{align*}
  	h(F(a_1,\ldots,a_n))&\geq h(F_i(a_1,\ldots,a_{n-1})a_n^i)-h(Q_i(a_1,\ldots,a_n))-\log 2\\
	  \begin{split}
	   	  &\geq ih(a_n)-h(F_i(a_1,\ldots,a_{n-1}))-2\log2\\
	   	  &\quad -(i-1)h(a_n)-\sum_{j=1}^{n-1}D_jh(a_j) - C_3
	  \end{split}\\
	  			&\ge h(a_n)-\sum_{j=1}^{n-1}2D_jh(a_j) - C_2
	  \end{align*}
where $C_3$ is the constant appearing when applying part (a) to the polynomial $Q_i$
and $C_2$ is another constant (depending on $C_3$ and the constant appearing when applying part (a) to the polynomial $F_i$). Clearly, all these constants depend only on the coefficients of $F$ and are effectively computable.
\end{enumerate}\end{proof}


	For every polynomial $f(x)\in\bQb[x]$ of degree $d\geq 2$, the canonical height $\hhat_f$ can
	be defined using a well-known trick of Tate (see \cite[Chapter 3]{Sil-ArithDS}):
	$$\hhat_f(a):=\lim_{m\to\infty}\frac{h(f^m(a))}{d^m}\text{, for all } a\in \bQb.$$

	We have the following properties: 
	\begin{lemma}\label{lem:basic cano}
		Let $f(x)\in\bQb[x]$ of degree $d\geq 2$. We have:
		\begin{itemize}
			\item [(a)] There is a constant $C_4$ depending only on $f$ such that $|\hhat_f(a)-h(a)|\leq C_4$ for every $a\in\bQb$.
			\item [(b)] $\hhat_f(f(a))=d\hhat_f(a)$ for every $a\in\bQb$.
			\item [(c)] $a$ is $f$-preperiodic if and only if $\hhat_f(a)=0$.
			\item [(d)] If $f$ is disintegrated and $g(x)\in \bQb[x]$ commutes with an iterate of $f(x)$ then $\hhat_f(g(a))=\deg(g)\hhat_f(a)$ for every $a\in\Qbar$.
		\end{itemize}
	\end{lemma}
	\begin{proof}
		For parts (a), (b) and (c): see \cite[Chapter 3]{Sil-ArithDS}. For
		part (d), Proposition~\ref{prop:all_g} 
		gives that $f$ and $g$ have a common iterate. Hence $\hhat_{f}=\hhat_g$ and the desired result follows from part (b). 
	\end{proof}

	The following inequality will be the key ingredient for the proof of
	Theorem~\ref{thm:informal}. It is essentially part (b) of Lemma~\ref{lem:important}
	where we replace $h$ by $\hhat_f$:
	\begin{cor}\label{cor:key inequality}
	Let $n\geq 1$ and  $F(X_1,\ldots,X_n)\in \bQb[X_1,\ldots,X_n]$ having degree $D\geq 1$ in $X_n$. Let $f(x)\in\Qbar[x]$ having degree at least 2.
		Write:
$$F(X_1,\ldots,X_n)=F_D(X_1,\ldots,X_{n-1})X_n^D+\ldots+F_0(X_1,\ldots,X_{n-1}).$$
	For $1\leq i\leq n-1$, let $D_i$ be the degree of $F$ in $X_i$. There
	exists a positive constant $C_5$ depending
			only on $F$ and $f$ such that
			for every $a_1,\ldots,a_n \in\bQb$ satisfying 
			$F(a_1,\ldots,a_n)=0$ and $F_i(a_1,\ldots,a_{n-1})\neq 0$
			for some $1\leq i\leq D$,
			we have:
			$$\hhat_f(a_n)\leq D_1\hhat_f(a_1)+\ldots+D_{n-1}\hhat_f(a_{n-1})+C_5.$$
	\end{cor}
	\begin{proof}
	This follows from part (b) of Lemma~\ref{lem:important} and 
	part (a) of Lemma~\ref{lem:basic cano}.
	\end{proof}
	
	We finish this section with the following remark about Proposition~\ref{prop:all_g}:
	\begin{remark}\label{rem:effective g}
	Recall the group $M(f^\infty)$ and the choice of $\tilde{f}$ from Proposition~\ref{prop:all_g}. We explain that both $M(f^\infty)$ and $\tilde{f}$
	could be determined effectively. In fact, after conjugating with an effectively 
	determined linear polynomial, we may assume that $f$ has the normal form:
	$$f(x)=x^d+a_{d-r}x^{d-r}+\ldots$$
	with $a_{d-r}\neq 0$ and $r\geq 2$. Then it is an easy exercise to show that
	$M(f^\infty)$ is a subgroup of the group of linear
	polynomials of the form $\mu x$ where $\mu$ is an 
	$r^{\text{th}}$ root of unity.
	
	To determine $\tilde{f}$, note that part (d) of Proposition~\ref{prop:all_g}
	gives that $d$ is a power of $\tilde{d}:=\deg(\tilde{f})$ and there
	are at most $|M(f^\infty)|$ choices of $\tilde{f}$. Let $K$ be the field obtained
	by adjoining all coefficients of $f$ to $\Q$. Then let $\tilde{K}$
	be the field obtained by adjoining
	all the coefficients of $\tilde{f}$ to $K$. We must have 
	$[\tilde{K}:K]\leq |M(f^\infty)|$ since $\sigma(\tilde{f})$ is another choice
	for $\tilde{f}$ for every $K$-embedding $\sigma$ of $\tilde{K}$ into $\Qbar$. 
	Together
	with $\hhat_f(\tilde{f}(m))=\tilde{d} \cdot \hhat_f(m)$, we have that the degree and 
	(Weil) height of $\tilde{f}(m)$ are bounded effectively in terms of $f$ and $m$
	for every integer $m$. Hence $\tilde{f}$ is effectively determined. 
	\end{remark}

	
	\section{Basic properties of the anomalous locus}\label{sec:elementary geometry}

	Throughout this section, we fix a disintegrated polynomial 
	$f(x)\in\Qbar[x]$ of degree $d\geq 2$, fix a positive integer $n$ and let
	$\varphi_n$ be the corresponding self-map of $(\bP^1)^n$ as in 
	Assumption~\ref{assume:same maps}. Fix
	an irreducible subvariety $X$ of $(\bP^1)^n$ of dimension
	$r$ satisfying $1\leq r\leq n-1$. Recall Definition~\ref{def:dyn special}
	and Definition~\ref{def:dyn anomalous}
	of special and anomalous subvarieties with respect to $\varphi_n$ and $X$.  In 
	this section we prove several 
	geometric properties of
	the set $X^{\oa}_{\varphi_n}$ which will be used repeatedly in the proofs
	of Theorem~\ref{thm:informal} and Theorem~\ref{thm:open} in
	the next 2 sections. 
	By the affine part of $(\bP^1)^n$, we mean the Zariski
	open subset $(\bA^1)^n=\bA^n$. More generally the affine part
	of a subset of $(\bP^1)^n$ is its intersection with
	$\bA^n$. 
	
	\begin{lemma}\label{lem:1st easy}
	If $X$ is contained in a proper special subvariety $Z$ of
	$(\bP^1)^n$ then $X^{\oa}_{\varphi_n}=\emptyset$. If the affine part
	of $X$ is empty then $X^{\oa}_{\varphi_n}=\emptyset$.
	\end{lemma}
	\begin{proof}
	The first assertion is immediate since $X$ itself is anomalous. For the second assertion,
	note that if the affine part of $X$ is empty then $X$ is contained in
	a special subvariety of the form $\infty\times (\bP^1)^{n-1}$.
	\end{proof}

	For any $j\geq r$ factors $\bP^1$ of $(\bP^1)^n$, we have
	the projection $(\bP^1)^n\to(\bP^1)^{j}$. The next property allows
	us to 
	assume that the image of $X$ under this projection
	has dimension $r$.
	\begin{lemma}\label{lem:projection r+1}
		Pick $j\geq r$ distinct factors $\bP^1$ of $(\bP^1)^n$ and let $\pi$
		denote the corresponding projection $(\bP^1)^n\to(\bP^1)^{j}$. If
		$\dim(\pi(X))<r$ then $X^{\oa}_{\varphi_n}=\emptyset$. 
	\end{lemma}
	\begin{proof}
		Without loss of generality, we assume that the chosen factors are the first
		$j$ factors $\bP^1$ in $(\bP^1)^n$.
		Since $\dim(\pi(X))<r$, for any point $\alpha\in \pi(X(\Qbar))$
		every irreducible component of $\pi^{-1}(\alpha)\cap X$
		has dimension at least 1 (see Mumford's book \cite[pp.~48]{MumfordRed}). By
		intersecting $X$ with the special variety $\alpha\times (\bP^1)^{n-j}$,
		we conclude that every irreducible component of $\pi^{-1}(\alpha)\cap X$
		is an anomalous subvariety of $X$. Since this holds for
		every $\alpha\in \pi(X(\Qbar))$, we have $X^{\oa}_{\varphi_n}=\emptyset$.
	\end{proof}

	Using the last two lemmas, for the rest of this section we \emph{make the extra assumption}
	that the affine part of $X$ is non-empty and the image of $X$ under the projection to
	$(\bP^1)^{j}$ has dimension $r$ for every choice of $j\geq r$ factors $\bP^1$. Then the next statement is immediate:
	\begin{cor}\label{cor:F^J}
	Let $J$ be any ordered subset of $I_n:=\{1,\ldots,n\}$ of size $r+1$ explicitly
	listed in increasing order as $(i_1,\ldots, i_{r+1})$. As before, we define 
	$\bP^J:=(\bP^1)^{r+1}$ equipped with the projection $\pi^J$ from $(\bP^1)^n$.
	Then there exists an irreducible
	polynomial $F^J(X_{i_1},\ldots,X_{i_{r+1}})$ such that
	the affine part $\pi^J(X)\cap\bA^J$ of $\pi^J(X)$ is defined by the equation
	$F^J(x_{i_1},\ldots,x_{i_{r+1}})=0$. The polynomial $F^J$ is unique up to multiplication
	by an element in $\Qbar^{*}$.
	\end{cor}

	We now fix a choice of $F^J$ for each ordered subset $J$ of $I_n$ as in 
	Corollary~\ref{cor:F^J}. We have the following result which allows
	us to apply Corollary~\ref{cor:key inequality} later.

	\begin{proposition}\label{prop:can use key}
	Assume that  $X^{\oa}_{\varphi_n}\cap\bA^n(\Qbar)$ is non-empty and pick
	any $\alpha=(\alpha_1,\ldots,\alpha_n)\in X^{\oa}_{\varphi_n}\cap\bA^n(\Qbar)$. Let $J$ be any ordered set $(i_1,\ldots,i_{r+1})$ as in Corollary~\ref{cor:F^J} and let $1\leq \ell\leq r+1$. Let $D\geq 0$ denote
	the degree of $X_{i_{\ell}}$ in $F^J$. Write:
	$$F^J=F^J_DX_{i_{\ell}}^D+\ldots+F^J_1X_{i_{\ell}}+F^J_0$$
	where $F^J_k$ for $0\leq k\leq D$ is a polynomial in
	the variables $X_{i_1},\ldots,\widehat{X_{i_\ell}},\ldots,X_{i_{r+1}}$. 
	Then there exists $1\leq k\leq  D$
	such that
	 $F^J_k(\alpha_{i_1},\ldots,\widehat{\alpha_{i_\ell}},\ldots,\alpha_{i_{r+1}})\neq 0$. 
	In particular,
	this implies $D\geq 1$.
	\end{proposition}

%

	\begin{proof}
		Without loss of generality, we may assume that $J=\{1,\ldots,r+1\}$ 
		with the usual order on the natural numbers
		and $\ell=r+1$. To simplify notation, write $(\bP^1)^{r+1}=(\bP^1)^J$ equipped
		with the projection $\pi:=\pi^J$ from $(\bP^1)^n$
		to the first $r+1$ coordinates. 

		We now assume that $F^J_k(\alpha_1\ldots,\alpha_r)=0$
		for every $1\leq k\leq D$ (which implies $F^J_0(\alpha_1,\ldots,\alpha_r)=0$ too) and we arrive at a contradiction. Write
		$\alpha'=(\alpha_1,\ldots,\alpha_r)$. Since the equation $F^J(x_1,\ldots,x_{r+1})=0$
		defines (the affine part of) $\pi(X)$, we have that $\pi(X)$ contains
		the curve $\alpha '\times\bP^1$. Consider the morphism  
		$\pi\mid_X:\ X\to \pi(X)$; we have that some irreducible component $Y$
		of $(\pi\mid_X)^{-1}\left(\alpha '\times \bP^1\right)$
		has dimension at least $1$ (see \cite[pp.~48]{MumfordRed}).
		In other words, we have:
		\begin{itemize}
			\item $Y\subseteq X\cap (\alpha '\times (\bP^1)^{n-r})$;
			\item $\dim(Y)>0.$
		\end{itemize}
		Therefore $Y$ is an anomalous subvariety of $X$ containing  $\alpha$, contradiction. 
		\end{proof}

		Now let $V$ be an irreducible $\varphi_n$-periodic hypersurface of $(\bP^1)^n$. We define next an embedding $e_V:(\bP^1)^{n-1}\to (\bP^1)^n$ such that $e_V\left((\bP^1)^{n-1}\right)=V$. According to Theorem~\ref{thm:MS theorem}, $V$ is defined by $x_i=\zeta$
		for some $1\leq i\leq n$ and periodic $\zeta$
		or $x_i=g(x_j)$ for some $1\leq i\neq j\leq n$ and $g\in \Qbar[x]$ commuting
		with an iterate of $f$. If $V$ is given by $x_i=\zeta$, we define:
		$$e_V(a_1,\ldots,a_{n-1}):=(a_1,\ldots,a_{i-1},\zeta,a_i,\ldots,a_{n-1}).$$
		If $V$ is given by $x_i=g(x_j)$ and $i<j$, we define:
		$$e_V(a_1,\ldots,a_{n-1})=(a_1,\ldots,a_{i-1},g(a_{j-1}),a_{i},\ldots,a_{n-1}),$$
while if $j<i$, we  define:
    $$e_V(a_1,\ldots,a_{n-1})=(a_1,\ldots,a_{i-1},g(a_{j}),a_{i},\ldots,a_{n-1}).$$
    Consider the self-map $\varphi_{n-1}$ on $\bA^{n-1}$. Since the point $\zeta$ is periodic 
    and $g$ commutes with an iterate of $f$, we obtain that
    $e_V$ maps $\varphi_{n-1}$-periodic subvarieties of $(\bP^1)^{n-1}$
    to $\varphi_n$-periodic subvarieties of $(\bP^1)^n$ (contained in $V$). Then it is also
    easy to prove that $e_V$ maps special subvarieties of $(\bP^1)^{n-1}$
    to special subvarieties of $(\bP^1)^n$. Later on, in the proof of 
    Theorem~\ref{thm:informal}, 
    we use $e_V$
    to identify irreducible components of $X\cap V$ with subvarieties of $(\bP^1)^{n-1}$ so 
    that
    we can use the induction hypothesis. In fact, we have the following:
		\begin{lemma}\label{lem:e_V and anomalous}
		Let $W$ be an irreducible subvariety and let $V$ be
		an irreducible $\varphi_n$-periodic hypersurface of $(\bP^1)^n$. Let
		$\alpha\in W^{\oa}_{\varphi_n}\cap V(\Qbar)$. Let $W'$ be an irreducible component of
		$W\cap V$ containing $\alpha$. Since $e_V^{-1}(W')$ is an irreducible
		subvariety of $(\bP^1)^{n-1}$, we can
		define the set $e_V^{-1}(W')^{\oa}_{\varphi_{n-1}}$ as before;
		 then $e_V^{-1}(\alpha)\in e_V^{-1}(W')^{\oa}_{\varphi_{n-1}}$. 
	\end{lemma}
	\begin{proof}
		Since $W^{\oa}_{\varphi_n}\neq \emptyset$, we have that $W$ is not contained in $V$. And since
		$W\cap V\neq \emptyset$, every irreducible component of
		$W\cap V$ has dimension $\dim(W)-1$. 

		
		We prove the lemma by contradiction: assume there exists an 
		anomalous subvariety $Y$ of $e_V^{-1}(W')$ containing $e_V^{-1}(\alpha)$. Hence there exists
		a special subvariety $Z$ of $(\bP^1)^{n-1}$ satisfying the following  condition:
		\begin{equation}
Y\subseteq e_V^{-1}(W')\cap Z\text{ and }\dim(Y) > \max\{0,\dim(e_V^{-1}(W'))+\dim(Z)-(n-1)\}.
	  \end{equation}
		
		We have $\dim(e_V^{-1}(W'))=\dim(W')=\dim(W)-1$, $\dim(e_V(Y))=\dim(Y)$ and $\dim(e_V(Z))=\dim(Z)$. Together with
		the previous condition, we have:
		\begin{equation}
			e_V(Y)\subseteq W\cap e_V(Z)\text{ and }\dim(e_V(Y))>
			\max\{0,\dim(W)+\dim(e_V(Z))-n\}.
		\end{equation}
		
		By the discussion in the paragraph before this lemma, we have that $e_V(Z)$
		is a special subvariety of $(\bP^1)^n$. Therefore $e_V(Y)$ is an anomalous 
		subvariety of $W$.  Since $\alpha\in e_V(Y)$,
		we get a contradiction.
	\end{proof}

We conclude this section with the following useful fact: 



\begin{cor}
\label{rem:go up}
Let $W\subseteq (\bP^1)^n$ be an irreducible subvariety, let $V\subseteq (\bP^1)^n$ be an irreducible $\varphi_n$-periodic hypersurface intersecting $W$ properly, and let $W'$ be an irreducible component of $W\cap V$.
\begin{itemize}
\item[(a)] If $Y\subseteq (\bP^1)^{n-1}$ is a $\varphi_{n-1}$-anomalous subvariety of $e_V^{-1}(W')$, then $e_V(Y)\subseteq (\bP^1)^n$ is a $\varphi_n$-anomalous subvariety of both $W'$ and of $W$.
\item[(b)] If $Z\subseteq V$ is a $\varphi_n$-special subvariety and if $\tilde{Y}\subseteq W'\cap Z$ is a $\varphi_n$-anomalous subvariety of $W$
satisfying $\dim(\tilde{Y})>\max\{0,\dim(W)+\dim(Z)-n\}$ then $e_V^{-1}(\tilde{Y})$ is a $\varphi_{n-1}$-anomalous subvariety of $e_V^{-1}(W')$. 
\item [(c)] If $V$ is given by the equation $x_i=\zeta$ then part (b) holds
without the assumption $Z\subseteq V$.
\end{itemize}
\end{cor}
\begin{proof}
	Part (a) follows immediately from the proof of Lemma~\ref{lem:e_V and anomalous}. 

For part (b), the condition $Z\subseteq V$ makes sure that $e_V^{-1}(Z)$
	remains a special subvariety of $(\bP^1)^{n-1}$. We have 
	 $e_V^{-1}(\tilde{Y})\subseteq e_V^{-1}(W')\cap e_V^{-1}(Z)$ and
	 use $\dim(e_V^{-1}(W'))=\dim(W)-1$ to obtain:
	$$\dim\left(e_V^{-1}(\tilde{Y})\right) >\max\{0,\dim(e_V^{-1}(W'))+\dim(e_V^{-1}(Z))-(n-1)\}.$$
	This proves that $e_V^{-1}(\tilde{Y})$ is a special subvariety of
	$e_V^{-1}(W')$.	

	For part (c), note that if $V$ is defined by $x_i=\zeta$
	then each irreducible component of $V\cap Z$ is also special. Hence the conclusion follows by applying part (b) to the special variety $Z'\subseteq V$ which is an irreducible component of $Z\cap V$ containing $\tilde{Y}$.
\end{proof}

	\section{Proof of Theorem~\ref{thm:informal} for $f\times\ldots\times f$}\label{sec:proof}
	Throughout this section, we fix a  polynomial
	$f(x)\in \Qbar[x]$ of degree $d\geq 2$ which is not conjugated to $x^d$ or to $\pm C_d(x)$. Let $n$ be a positive integer 
	and let $\varphi_n$ be the corresponding self-map of $(\bP^1)^n$ as in Assumption~\ref{assume:same maps}. Let $X$ be an  
	irreducible subvariety of $(\bP^1)^n$ of dimension $r$. Our goal is
	to prove Theorem~\ref{thm:informal} asserting that the set
	 $$X^{\oa}_{\varphi_n}\cap \bigcup_V V$$
	has bounded height where $V$ ranges among all irreducible $\varphi_n$-periodic
	subvarieties of dimension $n-r$. Note that this is obviously true when $r=0$ (or when $r=n$); so,   
	we now proceed by induction. Let $r\in\{1,\dots, n-1\}$ and assume 
	Theorem~\ref{thm:informal} is valid for all varieties of dimension less than $r$ for any $n$.

	We may assume $X^{\oa}_{\varphi_n}\neq\emptyset$; in particular, $X$ is not contained in
	any proper special subvariety of $(\bP^1)^n$ (by Lemma~\ref{lem:1st easy}). Furthermore, it suffices to replace $X^{\oa}_{\varphi_n}$ by its affine part. In other words, we only need to 
	prove that the set $\left(X^{\oa}_{\varphi_n}\cap\bA^n\right)\cap \Per^{[r]}$ has bounded height. The reason is that every point in the ``non-affine part''  
	$$X^{\oa}_{\varphi_n}\setminus \bA^n$$
	is contained in $X\cap H$ where $H$ is a $\varphi_n$-periodic
	hypersurface of the form (after a possible rearrangement of coordinates):
	$(\bP^1)^{n-1}\times \infty$. We now use Lemma~\ref{lem:e_V and anomalous}
	and the embedding $e_H$ introduced there to apply the
	induction hypothesis for the irreducible components of
	$e_H^{-1}(X\cap H)\subseteq (\bP^1)^{n-1}$.

	From now on, we assume $X^{\oa}_{\varphi_n}\cap \bA^n\neq \emptyset$. Since there are only finitely many
	possible signatures (see
	Definition~\ref{def:signature}) for all irreducible $\varphi_n$-periodic subvarieties of dimension
	$n-r$, we fix a signature $\scrS$ consisting of the following data:
	\begin{itemize}
		\item A (possibly empty) subset $J$ of $I_n:=\{1,\ldots,n\}$ such
		that $|I_n\setminus J|\geq n-r$.
		\item A partition of $I_n\setminus J$ into $n-r$ non-empty subsets
		$J_1,\ldots,J_{n-r}$.
		\item For each $1\leq k\leq n-r$, a choice of an order $\prec$ on $J_{n-r}$.
		So we can describe the ordered set $J_k$ as:
			$$i_{k,1}\prec i_{k,2}\prec\ldots\prec i_{k,m_k}$$
			where $m_k:=|J_k|$. 
	\end{itemize}

	\begin{convention}\label{convention:sub}
	From now on, to avoid triple subscripts we denote the coordinate functions
	$x_{i_{k,j}}$ as $x_{k,j}$; hence $x_{i_{k,m_k}}$
	is denoted $x_{k,m_k}$.
	\end{convention}

	It suffices to prove that the set
	$$X_{\scrS}:=(X^{\oa}_{\varphi_n}\cap \bA^n)\cap \bigcup_V V$$
	has bounded height, where $V$ ranges over all irreducible $\varphi_n$-periodic subvarieties
	of dimension $n-r$ \emph{having signature $\scrS$}. 
	Identify $(\bP^1)^n:=(\bP^1)^J\times (\bP^1)^{J_1}\times\ldots\times (\bP^1)^{J_{n-r}}$. Such a $V$ is described by the following equations: 
	\begin{itemize}
		\item The equations $x_i=\zeta_i$ for $i\in J$, where $\zeta_i$ is $f$-periodic.
		\item For $1\leq k\leq n-r$, the equations $x_{k,2}=g_{k,2}(x_{k,1})$,\ldots, $x_{k,m_k}=g_{k,m_k}(x_{k,m_{k}-1})$
		where each $g_{k,i}\in \Qbar[x]$ (for $1\leq k\leq n-r$ and $2\leq i\leq m_k$) is a polynomial commuting with
		an iterate of $f$.  
	\end{itemize}		
	If for some $1\leq k\leq n-r$, we have $m_k\geq 2$ then we define: 
	$$D(V):=\min\{\deg(g_{k,m_k}):\ 1\leq k\leq n-r,\ m_k\geq 2\}.$$ 
	If $m_k=1$
	for every $k$, then $V$ is simply of the form $(\zeta_i)_{i\in J}\times (\bP^1)^{n-r}$. If that is the case, we define $D(V)=+\infty$.

	\begin{proposition}\label{prop:bounded degree}
		There exist positive constants $c_1$, $c_2$ depending only on $X$ and
		$\scrS$
		such that for every irreducible $\varphi_n$-periodic subvariety $V$ of
		dimension $n-r$ having signature $\scrS$, if $D(V)>c_2$ then
		the height of points in $\left(X^{\oa}_{\varphi_n}\cap \bA^n\right)\cap V$ 
		is bounded above by $c_1$.
	\end{proposition}
	\begin{proof}
		Let $V$ be defined by the equations $x_i=\zeta_i$ for $i\in J$ and
		$x_{k,i}=g_{k,i}(x_{k,i-1})$ for $1\leq k\leq n-r$ and 
		$2\leq i\leq m_k$ as before. 
		
		We prove first the case when $m_k=1$ for each $k$; this also gives insight to the general case. In this special case, without loss
		of generality we assume $J=\{1,\ldots,r\}$ so that $V=\zeta\times (\bP^1)^{n-r}$, where $\zeta=(\zeta_1,\dots, \zeta_r)$ such that each $\zeta_i$ is a periodic point for $f$. Let 
$$\alpha:=(\alpha_1,\dots, \alpha_n)\in \left(X^{\oa}_{\varphi_n}\cap \bA^n\right)\cap V(\Qbar).$$ 
In particular, $\alpha_j=\zeta_j$ for $j=1,\dots, r$. Now, for each $i=r+1,\dots, n$, we let $\pi_i$ be the projection of $(\bP^1)^n$ onto $(\bP^1)^{r+1}$
consisting of the $r+1$ coordinates $x_1,\dots, x_r, x_i$. Using Proposition~\ref{prop:can use key}, Corollary~\ref{cor:key inequality} 
and Lemma~\ref{lem:basic cano}, we obtain that $h(\alpha_i)$ is uniformly bounded, as
desired. 

From now on, we assume that $m_k\ge 2$ for some $k\in\{1,\dots, n-r\}$. 
The idea of our proof is that on $V$, the $n-r$ coordinate
		functions $x_{k,m_k}$ for $1\leq k\leq n-r$
		dominate the other $r$ coordinate functions in terms of height.
		However, due to the equations defining $X$, the coordinate functions
		$x_{k,m_k}$ for $1\leq k\leq n-r$ cannot dominate the other $r$ coordinates
		``too much''. This will allow us to prove that for each periodic variety $V$ of signature $\scrS$, and for each $(\alpha_1,\dots, \alpha_n)\in \left(X^{\oa}_{\varphi_n}\cap \bA^n\right)\cap V(\Qbar)$ the heights of $\alpha_{k,m_k}$ for $k=1,\dots, n-r$ are uniformly bounded, and thus, in turn this would yield that the height of each $\alpha_i$ for $i=1,\dots, n$ is uniformly bounded. Furthermore, using Lemma~\ref{lem:basic cano} (a), it suffices to prove that $\hhat_f\left(\alpha_{k,m_k}\right)$ is uniformly bounded independent of $V$.  We formalize this idea as follows.
		
		Define $\Gamma:=I_n\setminus \{i_{1,m_1},\ldots,i_{n-r,m_{n-r}}\}$ (i.e.~the set of $r$ indices that are ``dominated by the other $n-r$ indices''). For $1\leq k\leq n-r$, define:
		$$\Gamma_k:=\{i_{k,m_k}\}\cup \Gamma.$$
		
		Let $\alpha=(\alpha_1,\ldots,\alpha_n)\in (X^{\oa}_{\varphi_n}\cap\bA^n)\cap 
		V(\Qbar)$. By the equations defining $V$, the definition of $D(V)$ and 
		part (d) of Lemma~\ref{lem:basic cano}, we have (see Convention~\ref{convention:sub}):
		\begin{equation}\label{eq:using V}
			\hhat_f(\alpha_{k,m_k})\geq 
			D(V)\hhat_f(\alpha_{k,i})\text{, for all } 1\leq k\leq n-r\text{ and }  1\leq i\leq m_k-1
		\end{equation}
		Note that (\ref{eq:using V}) is vacuously true when $m_k=1$, so inequality (\ref{eq:using V}) yields:
		\begin{equation}\label{eq:dominant terms are large}
		n\sum_{k=1}^{n-r}\hhat_f(\alpha_{k,m_k}) \geq D(V)\sum_{k=1}^{n-r}\, \sum_{i=1}^{m_k-1}\hhat_f(\alpha_{k,i})
		\end{equation}
		
		For $1\leq k\leq n-r$, we consider the projection $\pi_k:=\pi^{\Gamma_k}$ from 
		$(\bP^1)^n$
		onto $(\bP^1)^{\Gamma_k}=(\bP^1)^{r+1}$. By Corollary~\ref{cor:F^J},
		there is an irreducible polynomial $F_k:=F^{\Gamma_k}$
		in $r+1$ variables such that $\pi_k(X)$ is defined by the equation
		$F_k=0$. 
		
		By Proposition~\ref{prop:can use key} and Corollary~\ref{cor:key inequality},
		there exist positive constants $c_3$ and $c_4$ depending only on 
		$X$ and $\scrS$ such that:
		\begin{equation}\label{eq:use key}
			\hhat_f(\alpha_{k,m_k})\leq c_3\sum_{i\in \Gamma} \hhat_f(\alpha_i) + c_4.
		\end{equation}
for all $k=1,\dots, n-r$. By the equations defining $V$, we have $\alpha_i=\zeta_i$ is periodic for $i\in J$. Hence (\ref{eq:use key}) and part (c) of Lemma~\ref{lem:basic cano} give:
		\begin{equation}
			\hhat_f(\alpha_{k,m_k})\leq c_3\sum_{\ell=1}^{n-r}\,\sum_{i=1}^{m_{\ell}-1} \hhat_f(\alpha_{\ell,i}) + c_4.
		\end{equation}
for all $k=1,\dots, n-r$.	This yields:
		\begin{equation}\label{eq:dominant terms are small}
			\sum_{k=1}^{n-r}\hhat_f(\alpha_{k,m_k})\leq nc_3\sum_{\ell=1}^{n-r}\,\sum_{i=1}^{m_{\ell}-1} \hhat_f(\alpha_{\ell,i}) + nc_4
		\end{equation}
		From (\ref{eq:dominant terms are large}) and (\ref{eq:dominant terms are small}), we
		have that if $D(V)\geq 2n^2c_3$ then:
		$$\sum_{k=1}^{n-r}\hhat_f(\alpha_{k,m_k})\leq 2nc_4,$$
and more generally:
$$\sum_{i=1}^n \hhat_f(\alpha_i) \le 2nc_4 + \frac{c_4}{c_3}.$$
		This finishes the proof of the proposition.
	\end{proof}

	\begin{proof}[End of the proof of Theorem \ref{thm:informal} for $f\times\ldots\times f$.] We need to prove that the
	set $X_{\scrS}$ has bounded height. Let $c_1$ be
	the positive constant in the conclusion of Proposition~\ref{prop:bounded degree}, 
	it suffices to show that the set:
	$$X_{\scrS,c_1}:=(X^{\oa}_{\varphi_n}\cap\bA^n)\cap \bigcup_V V$$
	has bounded height, where $V$ ranges over all irreducible $\varphi_n$-periodic
	subvarieties of dimension $n-r$ having signature $\scrS$ and $D(V)\leq c_1$. For 
	every such $V$, choose
	$1\leq\ell\leq n-r$ such that
	$\deg(g_{\ell,m_{\ell}}) = D(V)\leq c_1$. By part (d) of 
	Proposition~\ref{prop:all_g}, there are only finitely many such polynomials $g_{\ell,m_{\ell}}$. We let $H$ be the periodic hypersurface defined by $x_{\ell,m_\ell}=g_{\ell,m_{\ell}}\left(x_{\ell, m_{\ell-1}}\right)$; therefore there are only \emph{finitely many}
	possibilities for $H$. Since $V\subseteq H$, we can apply the induction hypothesis to each irreducible component of $e_H^{-1}(X\cap H)$  
	by using  Lemma~\ref{lem:e_V and anomalous}. Note that going from $e_H^{-1}(X\cap H)$
	to $X\cap H$ can increase the height by the factor $\deg(g_{\ell,m_{\ell}})$, but
	this is fine since there are only finitely many possibilities for $g_{\ell,m_{\ell}}$. 
	\end{proof}

  \section{Proof of Theorem~\ref{thm:open} for $f\times\ldots\times f$}\label{sec:open}
	
	Throughout this section, let $f(x)\in\Qbar[x]$ be a polynomial of degree
	$d\geq 2$ which is not conjugated to $x^d$ or to $\pm C_d(x)$, let $n$ be a positive integer, and let $\varphi_n$ be the diagonal action of $f$ on $(\bP^1)^n$
	as in Assumption~\ref{assume:same maps}. Let $X\subseteq (\bP^1)^n$ be a given irreducible variety 
	of dimension $r$ as in the previous section. Let $U:=U_X$ denote the union of
	all the anomalous subvarieties of $X$; we will prove that $U$ is Zariski closed in
	$X$. This is obvious when $X$ is a curve because in this case either $X$ is itself anomalous, or $U=\emptyset$. Therefore Theorem~\ref{thm:open}
	holds when $n=1,2$. We proceed by induction: fix $n\geq 3$ and assume that
	the conclusion of the theorem is valid for all smaller values of $n$. We may 
	assume
	$2\leq r\leq n-1$. By Lemma~\ref{lem:1st easy}, Lemma~\ref{lem:projection r+1}
	and their proof, we may assume that $X$ is not contained in
	any special subvariety, and that for any choice of $j\geq r$ factors $\bP^1$ 
	the image of the projection from $X$ to $(\bP^1)^j$ has dimension $r$.

For each subset $J\subseteq I_n$, we let $U_J$ be the union of all anomalous subvarieties of $X$ obtained by intersecting $X$ with a special subvariety of the form $\zeta\times (\bP^1)^J$, where $\zeta\in (\bP^1)^{n-|J|}$ (after a possible rearrangement of coordinates); next we show  that $U_J$ is closed. Indeed, we let $J':=I_n\setminus J$, and then we take the projection $\pi_{J'}:X\lra (\bP^1)^{J'}$. According to \cite[pp.~51]{MumfordRed}, the set of all $x\in X$ for which there exists an irreducible component of the fiber $\pi_{J'}^{-1}(\pi_{J'}(x))$ of dimension at least equal to $\max\{1, r+|J|-n+1\}$ passing through $x$ is closed; so, $U_J$ is a closed subset of $X$, as claimed. Let 
$$U_0:=\bigcup_{J\subseteq I_n}U_J;$$
then $U_0$ is a closed subset of $X$. We also let $U_{\infty}$ denote the
union of all anomalous subvarieties of $X$ that are contained in
a hypersurface of the form $\infty\times (\bP^1)^{n-1}$ (after
a possible rearrangement of coordinates).  By Corollary~\ref{rem:go up} (especially part (c)) and the induction
hypothesis, the set $U_{\infty}$ is Zariski closed in $X$. 

Now we let $Y$ be an anomalous subvariety of $X$ that is not contained in $U_0\cup U_\infty$. Let $Z$ be a special subvariety of $X$ such that $Y\subseteq X\cap Z$
and 
$$\dim(Y)>\max\{0,\dim(X)+\dim(Z)-n\}.$$
Write $\ell=\dim(Z)$. Without loss of generality, write
$Z=\zeta\times Z_0$ where $\zeta\in (\bP^1)^m(\Qbar)$ (the first $m$
coordinates) and $Z_0$ is a $\varphi_{n-m}$-periodic
subvariety of $(\bP^1)^{n-m}$ which projects dominantly onto each coordinate of $(\bP^1)^{n-m}$; we allow for the possibility that $m=0$, in which case $Z$ is simply a $\varphi_n$-periodic subvariety of $(\bP^1)^n$. Let $\delta:=\dim(Y)$; then $\delta\ge \max\{1, r+\ell-n+1\}$. There exist $\delta$ coordinates of $(\bP^1)^n$ such that $Y$ maps dominantly to $(\bP^1)^{\delta}$. This fact follows from the Implicit Function Theorem by considering a smooth point of $Y$. Without loss of generality, we assume these $\delta$ coordinates are the last $\delta$ coordinates of $(\bP^1)^n$. Furthermore,
since the image of $Y$ under the projection map $(\bP^1)^n\mapsto (\bP^1)^{\delta}$ is closed, it has to be the entire
$(\bP^1)^{\delta}$.

Partition $\{m+1,\ldots,n\}$
into $\ell$ ordered subsets $J_i$ such that,
writing $(\bP^1)^n:=(\bP^1)^m\times (\bP^1)^{J_1}\times \cdots \times (\bP^1)^{J_\ell}$, then $Z=\zeta\times C_1\times \cdots \times C_\ell$, where each $C_i\subseteq (\bP^1)^{J_i}$ is a periodic curve. For each $i=1,\dots, \ell$ we list the elements of $J_i$ as $j_{i,1}\prec \cdots \prec j_{i,m_i}$ (where $m_i:=|J_i|$). As before,
to avoid triple subscripts (such as $x_{j_{i,m_i}}$), we use $x_{i,s}$
to denote $x_{j_{i,s}}$. Each $C_i$ is defined by the equations
$x_{i,2}=g_{i,2}(x_{i,1}),\ldots,x_{i,m_i}=g_{i,m_i}(x_{i,m_i-1})$
where each $g_{i,j}$ (for $1\leq i\leq \ell$ and $2\leq j\leq m_i$) commutes with an iterate of $f$.



Since $Y$ projects onto the last $\delta$ coordinates $(\bP^1)^{\delta}$ of $(\bP^1)^n$, then each index $n-\delta+1\le j\le n$ is part of a different chain $J_i$. Furthermore, because for each $i=1,\dots, \ell$, and for each $j=2,\dots, m_i$ we have $x_{i,j}=g_{i,j}(x_{i,j-1})$ for some polynomial $g_{i,j}$ (commuting with $f$), we may 
 assume 
$$\{n-\delta+1,\dots, n\}\subseteq \{j_{1,m_1},\dots, j_{\ell,m_\ell}\}.$$
Because $Y$ projects onto the last $\delta$ coordinates $(\bP^1)^{\delta}$, 
and $(Y\cap \bA^n)\setminus U_0$ is non-empty open in $Y$, given any real number  $B>1$, there exists $(a_1,\dots, a_n)\in ((Y\cap\bA^n)\setminus U_0) (\Qbar)$ such that $h(a_{n-\delta+i})<1/B$ for each $i=1,\dots, \delta-1$, while $h(a_n)>B$. Let $S_B$
denote the set of all such points (which is actually Zariski dense in $Y$). 

Let $\Gamma_0\subseteq I_n$ consist of: 
\begin{itemize}
\item each $s=1,\dots, m$;
\item each $s\in \left(\cup_{i=1}^\ell J_i\right)\setminus \{j_{1,m_1},\dots, j_{\ell,m_\ell}\}$; 
\item each $s=n-\delta+1,...,n-1$.
\end{itemize}
In other words, $x_s$ for $s\in \Gamma_0$ is either one of the first $m$ ``constant coordinates'', or one of the ``dominated coordinates'' in the chains $J_i$,
or one of the $\delta-1$ coordinates $x_{n-\delta+1},...,x_{n-1}$ whose
valuations on $S_B$ have heights bounded by $1/B$. By construction, the above three sets are disjoint and therefore: $$|\Gamma_0|=m+(n-m-\ell)+(\delta-1)=n-\ell-1+\delta\ge r.$$
So we can fix $\Gamma$ to be a subset of $\Gamma_0$ of cardinality $r$. For each $i\in I_n\setminus \Gamma_0$, let $\Gamma_i:=\Gamma\cup\{i\}$ and consider the projection $\pi_i:X\lra (\bP^1)^{J_i}$. Since $\dim(\pi_i(X))=r$, then $\pi_i(X)\subset (\bP^1)^{\Gamma_i}$ is defined by  $F_i=0$ where
$F_i$ is a polynomial in the variables $x_k$ for $k\in \Gamma_i$. We write: 
$$F_i=\sum_{j=0}^{D_i}F_{i,j} x_i^j,$$
where $D_i=\deg_{x_i}F_i$ and each $F_{i,j}$ is a polynomial in the variables $x_k$ where $k\in\Gamma$.

Let $(a_1,\dots, a_n)\in S_B$ for some $B>1$. We claim that for every $i\in I_n\setminus \Gamma_0$, there exists $j=0, \cdots ,D_i$ for which $F_{i,j}(a_k)_{k\in\Gamma}\ne 0$. 
Otherwise, as explained in the proof of Proposition~\ref{prop:can use key},
the point $(a_1,\ldots,a_n)$
would be contained in an anomalous subvariety obtained
by intersecting $X$ with $(a_k)_{k\in\Gamma}\times (\bP^1)^{I_n\setminus \Gamma}$.
This would imply $(a_1,\ldots,a_n)\in U_0(\Qbar)$, violating the 
above definition of $S_B$. Now we use arguments similar to those in the proof of
Proposition~\ref{prop:bounded degree}. By Corollary~\ref{cor:key inequality}, there
exist positive constants $c_5$ and $c_6$ depending only on $X$
such that  
that for every $i\in I_n\setminus \Gamma_0$, 
\begin{equation}
\label{ineg 1}
\hhat_f(a_i)\le c_5\sum_{j\in \Gamma} \hhat_f(a_j) + c_6.
\end{equation}
There exists $i\in I_n$ such that $m_i\ge 2$ since otherwise $Z=\zeta\times (\bP^1)^{n-m}$ contradicting our assumption that $Y$ is not contained in $U_0$. 
Let $M$ be a positive integer to be chosen later. If $\deg(g_{i,m_i})>M$ for each $i=1,\dots, \ell$ such that $m_i\ge 2$, then
$$n\hhat_f(a_{i,m_i})>M \sum_{j=1}^{m_i-1}\hhat_f(a_{i,j}).$$
Using the fact that $(a_1,\dots, a_m)=\zeta$ is constant, and also that $\hhat(a_{n-\delta+1}),\ldots,\hhat_f(a_{n-1})$ are uniformly bounded by
$1/B <1$, we conclude that there
exists a positive constant $c_7$ (depending only on $n$ and $\zeta$)
such that:
\begin{equation}
\label{ineg 2}
n\sum_{i\in I_n\setminus \Gamma_0}\hhat(a_i) >M\sum_{j\in\Gamma}\hhat(a_j)-c_7.
\end{equation}
Now fix $M>n^2 c_5$, then $\hhat(a_i)$ for $i\in I_n\setminus \Gamma_0$ is bounded solely in terms of $M,n,c_5,c_6,c_7$. This contradicts the fact that $\hhat_f(a_n)>B$
once $B$ is chosen to be sufficiently large. In conclusion, there exists $i=1,\dots,  \ell$ such that $m_i\geq 2$ and $\deg(g_{i,m_i})\le M$. 

By Proposition~\ref{prop:all_g}, there are at most finitely many polynomials $g$ of degree bounded by $M$ which commute with an iterate of $f$. For each such $g$
there
are $n(n-1)$
periodic hypersurfaces in $(\bP^1)^n$
defined by $x_j=g(x_k)$ for $1\leq k\neq j\leq n$. Denote the collection of
all such hypersurfaces (for all choices of $g$) by $\mathcal{V}$.
We conclude that for every 
special subvariety $Z$ such that there exists a subvariety $Y\subseteq X\cap Z$
satisfying
$\dim(Y)>\max\{0,\dim(X)+\dim(Z)-n\}$
and $Y\not\subseteq U_0\cup U_{\infty}$ then $Z\subseteq V$
for some $V\in \mathcal{V}$. 

For every $V\in\mathcal{V}$, let $W_{V,i}$ for $1\leq i\leq n(V)$
be all the irreducible components of $X\cap V$ and let $e_V$ denote
the embedding associated to $V$ as in Section~\ref{sec:elementary geometry}. 
Let $U_V$ denote the union of the anomalous loci of
$e_V^{-1}(W_{V,1}),\ldots,e_V^{-1}(W_{V,n(V)})$. 
By Corollary~\ref{rem:go up} and the induction hypothesis,
we have that $U$ is exactly the Zariski closed set:
\begin{equation}
\label{structure equation}
U_0\bigcup U_\infty\bigcup _{V\in\mathcal{V}} e_V(U_V).
 \end{equation}
This concludes the proof that $X^{\oa}_{\varphi_n}$ is Zariski open in $X$. 

The second part of Theorem~\ref{thm:open} holds immediately
for the anomalous subvarieties contained in $U_0$ simply by the definition of $U_0$.
For the anomalous subvarieties contained in $U_\infty$, we apply 
Corollary~\ref{rem:go up} and the induction hypothesis for 
the irreducible components of $X\cap H$ where $H$ ranges
over all special hypersurfaces of the form $\infty\times (\bP^1)^{n-1}$ (after
a possible rearrangement of coordinates). Finally, for the anomalous subvarieties
contained in $\bigcup_{V\in \mathcal{V}} e_V(U_V)$, we apply
Corollary~\ref{rem:go up} and
the induction hypothesis for the irreducible components
$e_V^{-1}(W_{V,i})$ for $V\in \mathcal{V}$ and $1\leq i\leq n(V)$. This finishes
the proof of the theorem for the dynamics of $f\times\ldots\times f$.

	\section{More general dynamical systems}\label{sec:general}
	\subsection{Proof of Theorem~\ref{thm:informal} and Theorem~\ref{thm:open} in general}
%

	
	Our proof consists of two steps:
	\begin{itemize}
		\item [(I)] Reduce from $\varphi$ to a map of the
		form $\psi=\psi_1\times\ldots\times \psi_s$ 
		where $\psi_i$ is a coordinate-wise self-map
		of $(\bP^1)^{n_i}$ of the form 
		$w_i\times\ldots\times w_i$ for some $1\leq n_i\leq n$
		and disintegrated $w_i(x)\in \Qbar[x]$ such that
		$\sum_{i=1}^s n_i=n$ and $w_i$ and $w_j$ are inequivalent
		for $i\neq j$ (see Definition~\ref{def:relation}).
		\item [(II)] Consider maps having the same form like $\psi$ above.
	\end{itemize}
	The arguments in Step (II) above are essentially
	the same as those used in the proof of Theorem~\ref{thm:informal} and
	Theorem~\ref{thm:open} in the special case $f\times\ldots\times f$; one simply needs to introduce an extra layer of complexity in the notation used in the previous sections. Hence we 
	will
	explain Step (I) in detail and only briefly sketch the arguments for Step (II).
	Again, the main ingredient in Step (I) comes from \cite{MS2014}. We need the 
	following:
	
	\begin{definition}\label{def:relation}
	  For simplicity, an irreducible curve
	  in $(\bP^1)^2$ is called non-trivial if the projection
	  to each factor $\bP^1$ is non-constant.
	  Let $A_1(x),A_2(x)\in\Qbar[x]$ be disintegrated polynomials of degrees at least 
	  2.
	  We define $A_1 \approx A_2$ if the self-map $A_1\times A_2$ of $(\bP^1)^2$
	  admits a non-trivial irreducible periodic curve.
	\end{definition}
	Now let $A_1 \approx A_2$ as in Definition~\ref{def:relation}, then
	there exist a positive integer $N$, non-constant $p_1(x),p_2(x)\in\Qbar[x]$ and a disintegrated
	polynomial $w(x)\in\Qbar[x]$ such that:
	$A_i^N\circ p_i=p_i\circ w$ for $i=1,2$ (see \cite[Proposition~2.34]{MS2014}). In 
	other words, we have the
	commutative diagram:
	$$\begin{diagram}
				 		\node{(\bP^1)^2}\arrow{s,t}{p_1\times p_2}\arrow{e,t}{w\times w}
				 		\node{(\bP^1)^2}\arrow{s,r}{p_1\times p_2}\\
				 		\node{(\bP^1)^2}\arrow{e,b}{A_1^N\times A_2^N}\node{(\bP^1)^2}
				 \end{diagram}$$
	We say that the dynamical system $w\times w$ covers $A_1^N\times A_2^N$
	by the covering $p_1\times p_2$.  We can easily show that $\approx$ is an equivalence 
	relation
	on the set of disintegrated polynomials of degrees as least $2$ (see also 
	\cite[Proposition 2.11]{MS2014}) as follows. Suppose we also have $A_2\approx A_3$.
	Then $A_2^N\approx A_3^N$, hence $w\approx A_3$ by using the
	covering $p_2\times \id$. Applying the (existence of the) above commutative 
	diagram for
	the pair $(w,A_3^N)$ instead of the previous pair $(A_1,A_2)$, there exists
	a positive integer $K$ such that $w^K\times A_3^{NK}$ is covered by
	$W\times W$ for some disintegrated polynomial $W$. Hence 
	the dynamical system
	$A_1^{NK}\times A_2^{NK}\times A_3^{NK}$ on $(\bP^1)^3$
	is covered by $W\times W\times W$. By using a non-trivial periodic
	curve under $W\times W$, we have that $A_1^{NK}\times A_3^{NK}$ admits a
	non-trivial periodic curve. Hence $A_1\approx A_3$. 
	
	Applying the above arguments inductively, we can also show that given disintegrated 
	polynomials
	$A_1,\ldots,A_{\ell}$ in the same equivalence class, there exist
	a positive integer $N$, 
	non-constant polynomials $p_1,\ldots,p_{\ell}$
	and a disintegrated polynomial $w$ such that $A_i^N\circ p_i=p_i\circ w$
	for every $1\leq i\leq \ell$. In other words,
	the dynamical system $w\times\ldots\times w$ covers 
	the system $A_1^N\times\ldots\times A_n^N$
	by the covering $p_1\times\ldots\times p_n$. 
	
	Now given disintegrated polynomials $f_1,\ldots,f_n$,  
	let $\varphi=f_1\times\ldots\times f_n$, and let $s$ denote the number of (distinct) equivalence
	classes corresponding to $f_1,\ldots,f_n$ (under the equivalence relation  $\approx$). Let $n_1,\ldots,n_s$
	denote the number of distinct $f_i$'s belonging in each of these classes  (hence $n_1+\ldots+n_s=n$). 
	There exist a positive integer $N$, 
	non-constant $p_1,\ldots,p_n\in\Qbar[x]$ and disintegrated $w_1,\ldots,w_s\in\Qbar[x]$
  such that the following holds. For $1\leq i\leq s$, let
  $\psi_i$ be the self-map $w_i\times\ldots\times w_i$ on $(\bP^1)^{n_i}$.
  Let $\eta=p_1\times\ldots\times p_n$ as a self-map on $(\bP^1)^n$.
  After a possible rearrangement of the polynomials $f_1,\ldots,f_n$, we
  have the commutative diagram:
  $$\begin{diagram}
				 		\node{(\bP^1)^{n_1}\times\ldots\times 
				  (\bP^1)^{n_s}}\arrow{s,t}{\eta}\arrow{e,t}{\psi_1\times\ldots\times\psi_s}
				 		\node{(\bP^1)^{n_1}\times\ldots\times 
				  (\bP^1)^{n_s}}\arrow{s,r}{\eta}\\
				 		\node{(\bP^1)^n}\arrow{e,b}{\varphi^N}\node{(\bP^1)^n}
				 \end{diagram}$$
  
	Since the statements of Theorem~\ref{thm:informal}
	and Theorem~\ref{thm:open} are unchanged
	when we replace $f_i$ by $f_i^N$ for $1\leq i\leq n$, we may assume
	$N=1$. 
	Write $\psi:=\psi_1\times\ldots\times\psi_s$. By \cite[Proposition 2.21]{MS2014},
	every irreducible $\psi$-periodic subvariety $V$ of $(\bP^1)^n=(\bP^1)^{n_1}\times\ldots\times(\bP^1)^{n_s}$
	has the form $V_1\times\ldots\times V_s$
	where each $V_i$ is a $\psi_i$-periodic subvariety of $(\bP^1)^{n_i}$ for
	$1\leq i\leq s$. Since $\eta$ is a finite and coordinate-wise morphism, it satisfies
	the following properties:
	\begin{itemize}
		\item [(i)] $\eta$ maps $\psi$-periodic subvarieties to $\varphi$-periodic subvarieties. Some irreducible component of the inverse image of
		a $\varphi$-periodic subvariety under $\eta$ is $\psi$-periodic. 
		\item [(ii)] The same conclusion in (i) remains valid for
		$\varphi$-special and $\psi$-special subvarieties.
	  \item [(iii)] As a consequence of (ii), for every irreducible subvariety	  
	  $X$ of $(\bP^1)^n$ and every irreducible component $W$ 
	  of $\eta^{-1}(X)$, $\eta$ maps the $\psi$-anomalous locus of $W$ into the $\varphi$-anomalous
	  locus of $X$. Furthermore, we have:
	  $$\eta(\displaystyle\cup_W W^{\oa}_{\psi})=X^{\oa}_{\varphi}\ \text{and}\ 
	  \eta(\displaystyle\cup_W (W-W^{\oa}_{\psi}))=X-X^{\oa}_{\varphi}$$
	  where $W$ ranges over all irreducible components of $\eta^{-1}(X)$.
	  \end{itemize}
	  
	 The above properties of $\eta$ reduce $\varphi$
	 to maps of the form $\psi_1\times\ldots\times\psi_s$. This finishes Step (I). 
	 We briefly sketch Step (II).
	 
	 For Bounded Height, we proceed as follows. We define a $\psi$-signature $\scrS$ to 
	 be a collection consisting of
	 a signature $\scrS_i$ for each block $(\bP^1)^{n_i}$ for $1\leq i\leq s$.
	  A $\psi$-periodic subvariety $V=V_1\times\ldots \times V_s$ as above is said to 
	  have 
	  signature $\scrS$
	  if each $V_i$ has signature $\scrS_i$. Given an irreducible subvariety
	  $W$ of $(\bP^1)^n$ having dimension $r$, it suffices to show that the set $$W^{\oa}_{\psi}\cap \left(\cup_V V\right)$$  
	  has bounded height, where $V$ ranges over all irreducible $\psi$-periodic
	  subvarieties of dimension $n-r$ \emph{having the fixed signature $\scrS$}.
	  
	  Now all results in Section \ref{sec:elementary geometry} remain valid; so, in particular, 
	  once we assume $W^{\oa}_{\psi}\neq \emptyset$, the projection from $W$ to any 
	  $r+1$ factors is a hypersurface. Proposition~\ref{prop:bounded degree} remains 
	  valid with a similar proof (albeit with a more complicated system of notations
	  in order to deal with the different blocks $(\bP^1)^{n_i}$). Note that instead of
	  the canonical height $\hhat_f$ used in the proof of Proposition~\ref{prop:bounded degree}, here we use the canonical heights 
	  $\hhat_{w_i}$
	 for the coordinates inside the block $(\bP^1)^{n_i}$ for $1\leq i\leq s$. 
	 Finally, we apply the induction hypothesis as in the end of Section~\ref{sec:proof}.
	 This finishes the proof of Theorem~\ref{thm:informal}.
	 
	  For proving that $W^{\oa}_{\psi}$ is Zariski open in $W$, we proceed
	  as follows. Let $U$ denote the  union of $\psi$-anomalous subvarieties
	  of $W$. Define $U_0$ and $U_{\infty}$ as in Section~\ref{sec:open}. Let
	  $Y$ be a $\psi$-anomalous subvariety of $X$ and $Z$ a $\psi$-special
	  subvariety
	  such that $Y\subseteq X\cap Z$, $\dim(Y)>\max\{0,\dim(X)+\dim(Z)-n\}$
	  and $Y\not\subseteq U_0\cup U_{\infty}$. The same argument as in Section~\ref{sec:open}
	  (here we use the canonical heights $\hhat_{w_i}$
	  for each block $(\bP^1)^{n_i}$) 
	  shows that there is a finite collection of
	  $\psi$-periodic 
	  hypersurfaces $\mathcal{V}$ such that for every $Y$ and $Z$ as above, we have
	  $Z\subseteq V$ for some $V\in \mathcal{V}$. Then we apply the
	  induction hypothesis. This finishes the proof of Theorem~\ref{thm:open}.

	 \subsection{More general questions for rational maps}
	 We need the following:
	 \begin{definition}
		A rational function $f(x)\in \Qbar(x)$ of degree $d\geq 2$ is said to be
		disintegrated if it is not linearly conjugate to $x^d$, $\pm C_d(x)$ or
		a Latt\`es map.
	\end{definition}
	For the definition of Latt\`es maps, we refer the readers to \cite[Chapter 6]{Sil-ArithDS}. Questions about the arithmetic dynamics of $f_1\times\ldots\times f_n$
	where each $f_i$ is Latt\`es reduce to diophantine questions on  
	 products of elliptic curves. The Bounded Height Conjecture has also
	been studied in this context (see, for example \cite{Viada03}, \cite{H2}).
	 
	 
	 Now let $n$ be a positive integer and let $f_1(x),\ldots,f_n(x)$
	be \emph{disintegrated} rational maps of degrees at least 2. Let $\varphi$
	denote the coordinate-wise self-map $f_1\times\ldots\times f_n$
	on $(\bP^1)^n$. Let $I_n:=\{1,\ldots,n\}$ as before. Recall that for an ordered subset
	$J$ of $I_n$ listed as $i_1\prec\ldots\prec i_m$ where $m=|J|$,
	we let $\varphi^J$ denote the self-map $f_{i_1}\times\ldots\times f_{i_m}$
	on $(\bP^1)^J$. Let $X$ be an irreducible subvariety of $(\bP^1)^n$. We can define 
	$\varphi$-special, $\varphi$-pre-special, $\varphi$-anomalous
	and $\varphi$-pre-anomalous subvarieties and the sets $X_{\varphi}^{\oa}$ and $X_{\varphi}^{\oa,\pre}$ as in Definition~\ref{def:dyn special} and Definition~\ref{def:dyn anomalous}. We expect an affirmative answer
	for the following:
  
	\begin{question}\label{q:general}
	Let $n$ be a positive integer, let $f_1,\ldots,f_n\in\Qbar(x)$ be disintegrated rational functions of degrees at least 2 and let $\varphi$ be the associated
	self-map of $(\bP^1)^n$.  Let $X$ be an irreducible subvariety of dimension $r$ in $(\bP^1)^n$. 
	\begin{itemize}
		\item [(a)] Is it true that the set $X^{\oa}_{\varphi}$
		is Zariski open in $X$ and $X^{\oa}_{\varphi}\cap \Per_{\varphi}^{[r]}$ has bounded height? 
		\item [(b)] Is it true that the set $X^{\oa,\pre}_{\varphi}$
		is Zariski open in $X$ and $X^{\oa,\pre}_{\varphi}\cap \Pre_{\varphi}^{[r]}$ has bounded height?
	\end{itemize}
 \end{question}

	\begin{remark}\label{rem:rational f}
		Consider the case $f_1=\ldots=f_n=f\in\Qbar(x)$ and
		\emph{assume the Medvedev-Scanlon classification (Theorem~\ref{thm:MS theorem})
		is valid for $f$}. Then
		part (a) of Question~\ref{q:general} has
		an affirmative answer. We can prove this by essentially the same
		arguments as in the proof of Theorem~\ref{thm:informal}
		and Theorem~\ref{thm:open} for $f\times\ldots\times f$. Besides the Medvedev-Scanlon
		classification,
		the two 
		key results needed in those proofs are: part (d) of Proposition~\ref{prop:all_g}
		stating that for every given degree there are only finitely many
		$g$ commuting with an iterate of $f$ and part (b) of Lemma~\ref{lem:important}. 		
		In fact, we used the ``affine part'' of $X$ 
		in the above proofs solely for simplifying the notation. In the more general case, 
		we can instead work with the \emph{multi-homogeneous} polynomial $F^J$
		defining the hypersurface $\pi^J(X)$ in $(\bP^1)^J$ (for $|J|=r+1$). More details
		are given in the next two propositions. 
	\end{remark}	
		
	First we need a counterpart of Proposition~\ref{prop:all_g} for rational maps:
	\begin{proposition}\label{prop:rational all_g}
	Let $f\in\Qbar(x)$ be a disintegrated rational map. Then the following hold:
	\begin{itemize}
		\item [(a)] If $g\in\Qbar(x)$ has degree at least 2 and commutes with an
		iterate of $f$ then $g$ and $f$ have a common iterate.
		\item [(b)] The group $M(f^\infty)$ of all
		M\"obius maps commuting with an iterate of $f$
		is finite.
		\item [(c)] Assume that the Medvedev-Scanlon classification 
		(Theorem~\ref{thm:MS theorem}) is valid for $f$. In the collection
		of rational maps of degrees at least 2 commuting with an iterate of $f$, choose
		$\tilde{f}$ that has the smallest degree. We have:
		$$\left\{\tilde{f}^m\circ L:\ m\geq 0,\ L\in M(f^\infty)\right\}=\left\{L\circ\tilde{f}^m:\ m\geq 0,\ L\in M(f^\infty)\right\}$$
		and this is exactly the set of all rational maps commuting
		with an iterate of $f$.
	\end{itemize} 
	\end{proposition}
	\begin{proof}
		Part (a) is a well-known theorem of Ritt \cite{Ritt23} while part (c) could
		be proved by the same arguments used in \cite[Proposition~2.3(d)]{IMRN2013}.
	  For part (b), Levin's theorem \cite{Levin90} implies that
	  for a given degree there are at most finitely many rational maps
	  having the same iterate with $f$. Together with part (a),
	  we have that the set $\{L\circ f:\ L\in M(f^\infty)\}$ is
	  finite. Hence $M(f^\infty)$ is finite.
	\end{proof}	
		
	We now let $[x_i:y_i]$ be the homogeneous coordinate on the $i^{\text{th}}$
	factor of $(\bP^1)^n$ for $1\leq i\leq n$. Every point $P\in\bP^1(\Qbar)$
	is always written in homogeneous coordinate as $[a:b]$ with $b=1$
	or $b=0$ and $a=1$. Hence given any homogeneous polynomial $F(x,y)$,
	we have a well-defined value $F(P)$. Now let $F(X_1,Y_1,\ldots,X_n,Y_n)$
	be a multi-homogeneous polynomial that is homogeneous in each
	$[X_1:Y_1],\ldots,[X_n:Y_n]$. Let $D$ be the homogeneous degree of $F$
	in $[X_n:Y_n]$. Assume $D\geq 1$ and write:
	$$F=F_DX_n^D+F_{D-1}X_n^{D-1}Y_n+\ldots+F_0Y_n^D$$ 
	where each $F_j$ is a multi-homogeneous polynomial in $X_1,\ldots,Y_{n-1}$.
	We have the following:
	\begin{lemma}
		There exist positive constants $C_8$ and $C_9$ depending
		only on $F$ such that
		for every $(P_1,\ldots,P_n)\in (\bP^1)^n(\Qbar)$ 
		satisfying $F_j(P_1,\ldots,P_{n-1})\neq 0$
		for some $1\leq j\leq D$, we have:
		$$h(P_n)-C_8(h(P_1)+\ldots+h(P_{n-1}))-C_9\leq h(F(P_1,\ldots,P_n)).$$
	\end{lemma}
	\begin{proof}
		The inequality is obvious (for any positive $C_8$ and $C_9$) when 
		$P_n=\infty$. When $P_n\neq \infty$, we can use completely similar arguments
		used in the proof of part (b) of 
		Lemma~\ref{lem:important}.
	\end{proof}
		
	We conclude this paper with the following remark concerning part (a) of 
	Question~\ref{q:general} in general:
	\begin{remark}
\label{last general remark}
	 We now consider the most general case $\varphi=f_1\times\ldots\times f_n$. 
	 The arguments in Step (I) in the last subsection
	 actually work when $f_1,\ldots,f_n$ are disintegrated \emph{rational} maps. 
	 This boils down to \cite[Definition 2.20, Fact 2.25]{MS2014}; although the statement 
	 given
	 by Medvedev-Scanlon in \cite[Fact 2.25]{MS2014} only treats the polynomial case,
	 it is also known to be true for rational maps thanks to Medvedev's PhD thesis 
	 \cite[Theorem~10]{Alice-PhD}. Step (II) could be done as in the previous 
	 subsection where we decompose
	 $(\bP^1)^n$ into blocks $(\bP^1)^{n_1},\ldots,(\bP^1)^{n_s}$. Therefore
	 part (a) of Question~\ref{q:general} has
	 an affirmative answer if
	 the Medvedev-Scanlon classification (Theorem~\ref{thm:MS theorem})
	 holds for every disintegrated rational function $f(x)\in \Qbar(x)$. In an
	 ongoing joint work of the second author with Michael Zieve, there are
	 given reasons to expect that the Medvedev-Scanlon classification	should hold
	 in such generality.
	\end{remark}

\end{document}